\documentclass[12pt]{amsart}
\usepackage{amsmath, amssymb, epsfig}
\usepackage{bm}
\usepackage{mathrsfs}
\usepackage{color}
\usepackage{alg}
\usepackage{subfigure}

\makeatletter
\def\algspacing{\alg@unmargin}
\makeatother

\newlength{\algorithmwidth}
\theoremstyle{plain}
\newtheorem{theorem}{Theorem}[section]

\newtheorem{lemma}[theorem]{Lemma}

\theoremstyle{definition}

\theoremstyle{remark}
\newtheorem*{remark}{Remark}

\numberwithin{equation}{section}

\newcommand{\<}{\left\langle}
\renewcommand{\>}{\right\rangle}
\newcommand{\bigO}{\mathrm{O}}

\newcommand{\defby}{\overset{\mathrm{\scriptscriptstyle{def}}}{=}}

\def \E {\mathbb{E}}

\newcommand{\vct}[1]{\bm{#1}}
\newcommand{\mtx}[1]{\bm{#1}}

\begin{document}
\bibliographystyle{plain}
\setlength{\parindent}{0in}
\parskip 7.2pt

\title[]{Two-subspace Projection Method for Coherent Overdetermined Systems}
\author{Deanna Needell and Rachel Ward}
\date{\today}

\begin{abstract}
We present a Projection onto Convex Sets (POCS) type algorithm for solving systems of linear equations.  POCS methods have found many applications ranging from computer tomography to digital signal and image processing.  The Kaczmarz method is one of the most popular solvers for overdetermined systems of linear equations due to its speed and simplicity.  Here we introduce and analyze an extension of the Kaczmarz method which iteratively projects the estimate onto a solution space given from two randomly selected rows.  We show that this projection algorithm provides exponential convergence to the solution in expectation.  The convergence rate significantly improves upon that of the standard randomized Kaczmarz method when the system has coherent rows.  We also show that the method is robust to noise, and converges exponentially in expectation to the noise floor.  Experimental results are provided which confirm that in the coherent case our method significantly outperforms the randomized Kaczmarz method.
\end{abstract}

\maketitle

\section{Introduction}\label{sec:intro}
We consider a consistent system of linear equations of the form
$$
\mtx{A}\vct{x} = \vct{b},
$$
where $\vct{b}\in\mathbb{C}^m$ and $\mtx{A}\in\mathbb{C}^{m\times n}$ is a full-rank $m\times n$ matrix that is overdetermined, having more rows than columns ($m\geq n$).  When the number of rows of $\mtx{A}$ is large, it is far too costly to invert the matrix to solve for $\vct{x}$, so one may utilize an iterative solver such as the Projection onto Convex Sets (POCS) method, used in many applications of signal and image processing~\cite{CFMSS92,SS87:Applications}.  The Kaczmarz method is often preferred, iteratively cycling through the rows of $\mtx{A}$ and orthogonally projecting the estimate onto the solution space given by each row~\cite{K37:Angena}.  Precisely, let us denote by $\vct{a_1}$, $\vct{a_2}$, $\ldots$, $\vct{a_m}$ the rows of $\mtx{A}$ and $b_1$, $b_2$, $\ldots$, $b_m$ the coordinates of $\vct{b}$.  For simplicity, we will assume throughout that the matrix $\mtx{A}$ is \textit{standardized}, meaning that each of its rows has unit Euclidean norm; generalizations from this case will be straightforward.  Given some trivial initial estimate $\vct{x_0}$, the Kaczmarz method cycles through the rows of $\mtx{A}$ and in the $k$th iteration projects the previous estimate $\vct{x_k}$ onto the solution hyperplane of $\langle \vct{a_i}, \vct{x}\rangle = b_i$ 
where $i = k$ mod $m$, 
$$
\vct{x_{k+1}} = \vct{x_k} + (b_i - \langle \vct{a_i}, \vct{x_{k}} \rangle)\vct{a_i}.
$$

Theoretical results about the rate of convergence of the Kaczmarz method have been difficult to obtain, and most are based on quantities which are themselves hard to compute~\cite{DH97:Therate,G05:Onthe}. Even more importantly, the method as we have just described depends heavily on the ordering of the rows of $\mtx{A}$.  A malicious or unlucky ordering may therefore lead to extremely slow convergence.  To overcome this, one can select the rows of $\mtx{A}$ in a \emph{random} fashion rather than cyclically~\cite{HM93:Algebraic,N86:TheMath}.  Strohmer and Vershynin analyzed a randomized version of the Kaczmarz method that in each iteration selects a row of $\mtx{A}$ with probability proportional to its Euclidean norm~\cite{SV09:Arand,SV06:Arandom}.  Thus in the standardized case we consider here, a row of $\mtx{A}$ is chosen uniformly at random.  This randomized Kaczmarz method is described by the following pseudocode.

\begin{algorithm}[thb]
\caption{Randomized Kaczmarz}
	\label{alg:rk}
\centering \fbox{
\begin{minipage}{.99\textwidth} 
\vspace{4pt}
\alginout{Standardized matrix $\mtx{A}$, vector $\vct{b}$}
{An estimation $\vct{x_k}$ of the unique solution $\vct{x}$ to $\mtx{A}\vct{x} = \vct{b}$
}
\vspace{8pt}\hrule\vspace{8pt}

\begin{algtab*}
Set $\vct{x_0}$.
	 	\hfill \{ Trivial initial approximation \} \\		 
$k \leftarrow 0$ \\

\algrepeat
	$k \leftarrow k + 1$ \\
	Select $ r \in \{1, 2, \ldots, n\}$
		\hfill \{ Randomly select a row of $\mtx{A}$ \} \\
		Set $ \vct{x_k} \leftarrow \vct{x_{k-1}} + (b_r - \langle \vct{a_r},\vct{x_{k-1}} \rangle)\vct{a_r}$
		\hfill \{ Perform projection \} \\
	
\end{algtab*}
\vspace{-10pt}
\end{minipage}}
\end{algorithm}

Note that this method as stated selects each row \textit{with replacement}, see~\cite{Valley} for a discussion on the differences in performance when selecting with and without replacement.  Strohmer and Vershynin show that this method exhibits exponential convergence in expectation~\cite{SV09:Arand,SV06:Arandom},
\begin{equation}\label{SV}
\E \|\vct{x_k} - \vct{x}\|_2^2 \leq \left( 1 - \frac{1}{R}\right)^k\|\vct{x_0} - \vct{x}\|_2^2,\quad\text{where}\quad R \defby \|\mtx{A}\|_F^2\|\mtx{A}^{-1}\|^2.
\end{equation}
Here and throughout, $\|\cdot\|_2$ denotes the vector Euclidean norm, $\| \cdot \|$ denotes the matrix spectral norm, $\|\cdot\|_F$ denotes the matrix Frobenius norm, and the inverse $\|\mtx{A}^{-1}\| = \inf\{M : M\|\mtx{A}\vct{x}\|_2 \geq \|\vct{x}\|_2 \text{ for all }\vct{x}\}$ is well-defined since $\mtx{A}$ is full-rank.  This bound shows that when $\mtx{A}$ is well conditioned, the randomized Kaczmarz method will converge exponentially to the solution in just $\bigO(n)$ iterations (see Section 2.1 of~\cite{SV09:Arand} for details).  The cost of each iteration is the cost of a single projection and takes $\bigO(n)$ time, so the total runtime is just $\bigO(n^2)$.  This is superior to Gaussian elimination which takes $\bigO(mn^2)$ time, especially for very large systems.  The randomized Kaczmarz method even substantially outperforms the well-known conjugate gradient method in many cases~\cite{SV09:Arand}.  

Leventhal and Lewis show that for certain probability distributions, the expected rate of convergence can be bounded in terms of other natural linear-algebraic quantities.  They propose generalizations to other convex systems~\cite{leventhal2010randomized}.
Recently, Chen and Powell proved that for certain classes of random matrices $\mtx{A}$, the randomized Kaczmarz method convergences exponentially to the solution not only in expectation but also almost surely~\cite{CP12:almost}.   

In the presence of noise, one considers the possibly inconsistent system $\mtx{A}\vct{x} + \vct{w} \approx \vct{b}$ for some error vector $\vct{w}$.  In this case the randomized Kaczmarz method converges exponentially fast to the solution within an error threshold~\cite{N10:rknoisy},
\begin{equation}\label{eq:noise} 
\E\|\vct{x_k} - x\|_2 \leq \left( 1 - \frac{1}{R}\right)^{k/2}\|\vct{x_0} - x\|_2 + \sqrt{R}\|\vct{w}\|_{\infty},
\end{equation}
where $R$ the the scaled condition number as in \eqref{SV} and $\|\cdot\|_{\infty}$ denotes the largest entry in magnitude of its argument.  This error is sharp in general~\cite{N10:rknoisy}.  Modified Kaczmarz algorithms can also be used to solve the least squares version of this problem, see for example~\cite{drineas2007faster,ENP10:semi,HN90:Onthe,censor1983strong} and the references therein.

\subsection{Coherent systems}

Although the convergence results for the randomized Kaczmarz method hold for any consistent system, the factor $\frac{1}{R}$ in the convergence rate may be quite small for matrices with many correlated rows.  
Consider for example the reconstruction of a bandlimited function from nonuniformly spaced samples, as often arises in geophysics as it can be physically challenging to take uniform samples.  Expressed as a system of linear equations, the sampling points form the rows of a matrix $\mtx{A}$; for points that are close together, the corresponding rows will be highly correlated.  

To be precise, we examine the \textit{coherence} of a standardized matrix $\mtx{A}$ by defining the quantities
\begin{equation}\label{mus}
\Delta = \Delta(\mtx{A}) = \max_{j\ne k}|\langle \vct{a_j}, \vct{a_k}\rangle| \quad{and}\quad
\delta = \delta(\mtx{A}) = \min_{j\ne k}|\langle \vct{a_j}, \vct{a_k}\rangle|.
\end{equation}
Note that because $\mtx{A}$ is standardized, $0 \leq \delta \leq \Delta \leq 1$. 
It is clear that when $\mtx{A}$ has high coherence parameters, $\|\mtx{A}^{-1}\|$ is very small and thus the factor $R$ in~\eqref{SV} is also small, leading to a weak bound on the convergence.  Indeed, when the matrix has highly correlated rows, the angles between successive orthogonal projections are small and convergence is stunted.   We can explore a wider range of orthogonal directions by looking towards solution hyperplanes spanned by  \emph{pairs} of rows of $\mtx{A}$.  We thus propose a modification to the randomized Kaczmarz method where each iteration performs an orthogonal projection onto a two-dimensional subspace spanned by a randomly-selected pair of rows. We point out that the idea of projecting in each iteration onto a subspace obtained from multiple rows rather than a single row has been previously investigated numerically, see e.g.~\cite{FS95,CFMSS92}.

With this as our goal, a single iteration of the modified algorithm will consist of the following steps.  Let ${\vct{x_k}}$ denote the current estimation in the $k$th iteration.

\begin{itemize}
\item Select two distinct rows $\vct{a_r}$ and $\vct{a_s}$ of the matrix $\mtx{A}$ at random
\item Compute the translation parameter $\varepsilon$
\item Perform an intermediate projection: $\vct{y} \leftarrow \vct{x_k} + \varepsilon(b_r - \langle \vct{x_k}, \vct{a_r}\rangle)\vct{a_r}$
\item Perform the final projection to update the estimation: $\vct{x_{k+1}} \leftarrow \vct{y} + (b_s - \langle \vct{y}, \vct{a_s}\rangle)\vct{a_s}$
\end{itemize}

In general, the optimal choice of $\varepsilon$ at each iteration of the two-step procedure corresponds to subtracting from $\vct{x_k}$ its orthogonal projection onto the solution space $\{\vct{x}: \langle \vct{a_r}, \vct{x}\rangle = b_r \text{ and } \langle \vct{a_s}, \vct{x}\rangle = b_s\}$, which motivates the name two-subspace Kaczmarz method. 
By \emph{optimal choice} of $\varepsilon$, we mean the value $\varepsilon_{opt}$ minimizing the residual $\|\vct{x} - \vct{x_{k+1}}\|_2^2$.   Expanded, this reads  
$$
\|\vct{x} - \vct{x_{k+1}}\|_2^2 = \|\varepsilon(b_r - \<\vct{x_k}, \vct{a_r}\>)(\vct{a_r} - \<\vct{a_s}, \vct{a_r}\>\vct{a_s}) + \vct{x_k} - \vct{x} + (b_s - \<\vct{x_k}, \vct{a_s}\>)\vct{a_s}\|_2^2.
$$

Using that the minimizer of $\|\gamma \vct{w} + \vct{z}\|_2^2$ is $\gamma = -\frac{\<\vct{w},\vct{z}\>}{\|\vct{w}\|_2^2}$, we see that 
\begin{equation*}
\varepsilon_{opt} = \frac{-\<\vct{a_r} - \<\vct{a_s}, \vct{a_r}\>\vct{a_s} , \vct{x_k} - \vct{x} + (b_s - \<\vct{x_k},\vct{a_s}\>)\vct{a_s}\>}{(b_r - \<\vct{x_k}, \vct{a_r}\>)\|\vct{a_r} - \<\vct{a_s}, \vct{a_r}\>\vct{a_s}\|_2^2}.
\end{equation*}
Note that the unknown vector $\vct{x}$ appears in this expression only through its observable inner products, and so $\varepsilon_{opt}$ is computable.
After some algebra, one finds that the two-step procedure with this choice of $\varepsilon_{opt}$ can be re-written as follows~\cite{NW12:2srk}.

\begin{algorithm}[thb]
\caption{Two-subspace Kaczmarz}
	\label{alg:r2k}
\centering \fbox{
\begin{minipage}{.99\textwidth} 
\vspace{4pt}
\alginout{Matrix $\mtx{A}$, vector $\vct{b}$}
{An estimation $\vct{x_k}$ of the unique solution $\vct{x}$ to $\mtx{A}\vct{x} = \vct{b}$
}
\vspace{8pt}\hrule\vspace{8pt}

\begin{algtab*}
Set $\vct{x_0}$.
	 	\hfill \{ Trivial initial approximation \} \\		 
$k \leftarrow 0$ \\

\algrepeat
	$k \leftarrow k + 1$ \\
	Select $ r, s \in \{1, 2, \ldots, n\}$
		\hfill \{ Select two distinct rows of $\mtx{A}$ uniformly at random \} \\
		Set $ \mu_k \leftarrow \langle \vct{a_{r}}, \vct{a_{s}}\rangle$
		\hfill \{ Compute correlation \} \\
		Set $ \vct{y_k} \leftarrow \vct{x_{k-1}} + (b_s - \langle \vct{x_{k-1}}, \vct{a_s}\rangle)\vct{a_s}$
		\hfill \{ Perform intermediate projection \} \\
		Set $ \vct{v_k} \leftarrow \frac{\vct{a_r} - \mu_k \vct{a_s}}{\sqrt{1-|\mu_k|^2}}$
		\hfill \{ Compute vector orthogonal to $\vct{a_s}$ in direction of $\vct{a_r}$ \} \\
		Set $ \beta_k \leftarrow \frac{b_r - b_s\mu_k}{\sqrt{1-|\mu_k|^2}}$
		\hfill \{ Compute corresponding measurement \}\\
	$\vct{x_k} \leftarrow \vct{y_k} + (\beta_k - \langle \vct{y_k}, \vct{v_k}\rangle)\vct{v_k}$
		\hfill \{ Perform projection \} \\	
	
\end{algtab*}
\vspace{-10pt}
\end{minipage}}
\end{algorithm}


Our main result shows that the two-subspace Kaczmarz algorithm provides the same exponential convergence rate as the standard method in general, and substantially improved convergence when the rows of $\mtx{A}$ are coherent~\cite{NW12:2srk}.  Figure \ref{fig:A} plots two iterations of the one-subspace random Kaczmarz and compares this to a single iteration of the two-subspace Kaczmarz algorithm.

\begin{figure}[ht]
\centering
\subfigure[]{
   \includegraphics[scale=0.5] {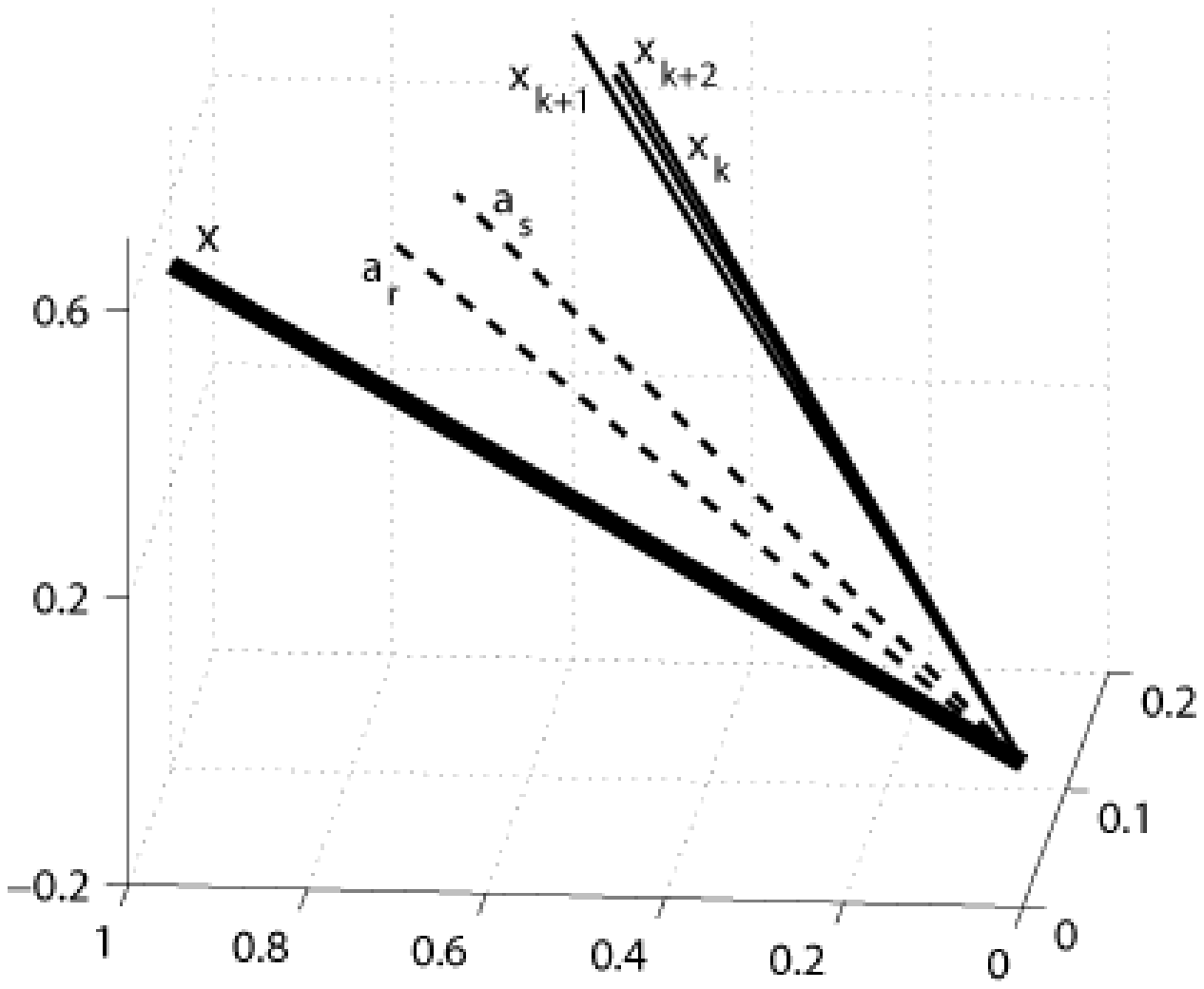}
 }
 \subfigure[]{
   \includegraphics[scale=0.5] {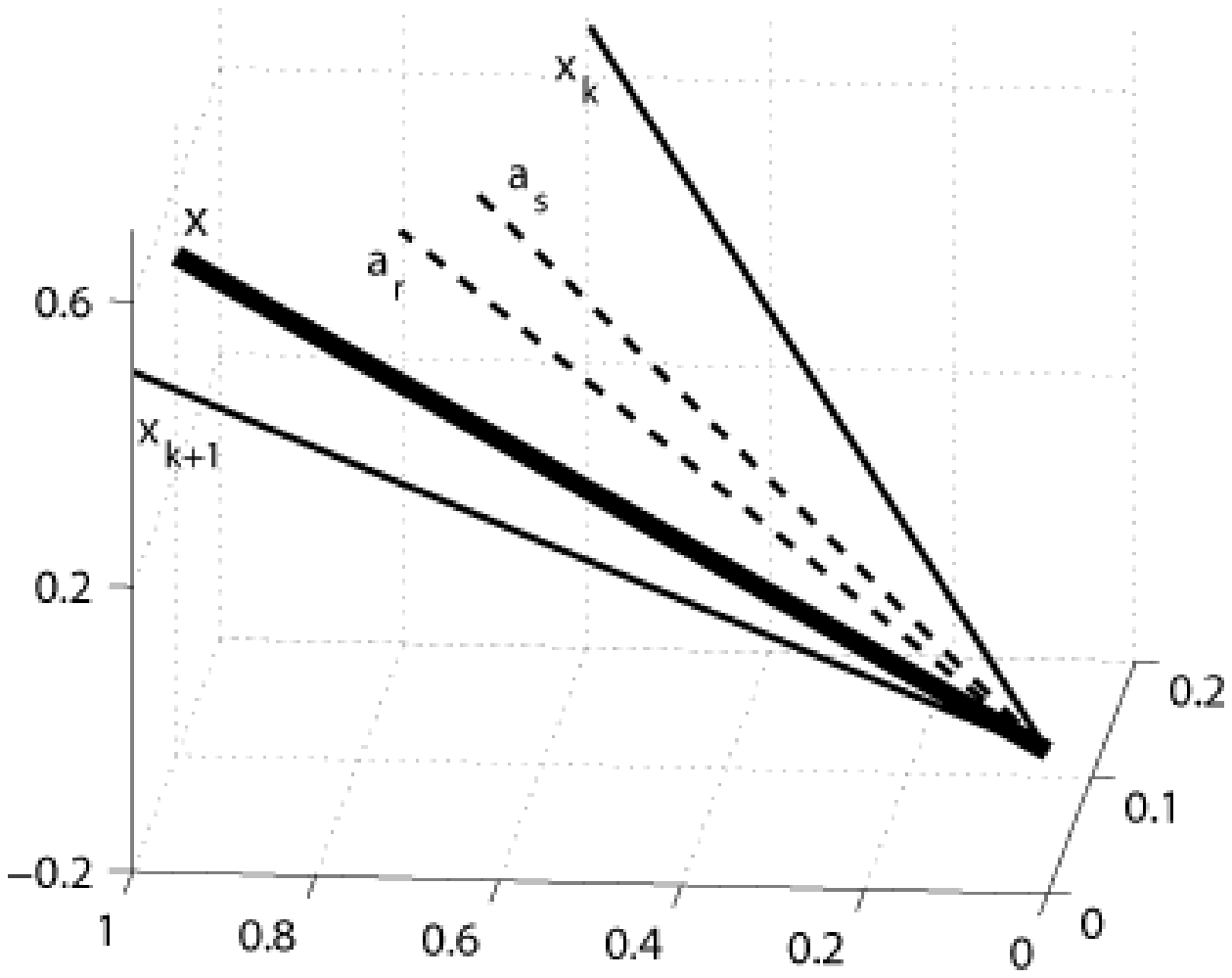}
 }
\caption{For coherent systems,  the one-subspace randomized Kaczmarz algorithm (a) converges more slowly than the two-subspace Kaczmarz algorithm (b).} \label{fig:A}
\end{figure}

%


\begin{theorem}\label{thm:main}
Let $\mtx{A}$ be a full-rank standardized matrix with $n$ columns and  $m > n$ rows and suppose $\mtx{A}\vct{x} = \vct{b}$.  Let $\vct{x_k}$ denote the estimation to the solution $\vct{x}$ in the $k$th iteration of the two-subspace Kaczmarz method.  Then
$$
\mathbb{E}\|\vct{x} - \vct{x_k}\|_2^2 \leq \left(\left(1 - \frac{1}{R}\right)^2 - \frac{D}{R}\right)^k\|\vct{x} - \vct{x}_{0}\|_2^2 ,
$$

where $D = \min\Big\{ \frac{\delta^2(1-\delta)}{1+\delta}, \frac{\Delta^2(1-\Delta)}{1+\Delta} \Big\} $, $\Delta$ and $\delta$ are the coherence parameters~\eqref{mus}, and $R = \| \mtx{A} \|_{F}^2 \|\mtx{A}^{-1}\|^2$ denotes the scaled condition number.
\end{theorem}

\begin{remarks}
{\bfseries 1. }When $\Delta = 1$ or $\delta = 0$ we recover the same convergence rate as provided for the standard Kaczmarz method~\eqref{SV} since the two-subspace method utilizes two projections per iteration. 

{\bfseries 2. }The bound presented in Theorem~\ref{thm:main} is a pessimistic bound.  Even when $\Delta = 1$ or $\delta = 0$, the two-subspace method improves on the standard method if  any rows of $\mtx{A}$ are highly correlated (but not equal).  This is evident in the proof of Theorem~\ref{thm:main} in Section~\ref{sec:proofs} but we present this bound for simplicity.  See also Section~\ref{sec:mods} for more details on improved convergence bounds.

\end{remarks}

Figure~\ref{fig:D} shows the value of $D$ of Theorem \ref{thm:main} for various values of $\Delta$ and $\delta$.  This demonstrates that in the best case (when $\delta \approx \Delta \approx 0.62$), the convergence rate is improved by at least a factor of 0.1.  

  \begin{figure}[h!]
\begin{center}
\includegraphics[width=3in]{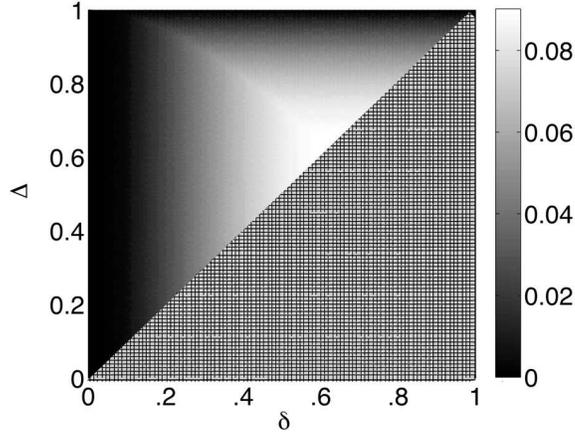}    
\end{center}
\caption{A plot of the improved convergence factor $D$ as a function of the coherence parameters $\delta$ and $\Delta \geq \delta$. }\label{fig:D}
\end{figure}

\subsection{Organization} The remainder of the report is organized as follows. In Section~\ref{sec:proofs} we state and prove the main lemmas which serve as the proof of Theorem~\ref{thm:main}.  Section~\ref{sec:noise} discusses the two-subspace Kaczmarz method in the presence of noise and shows that in this case the method exhibits exponential convergence to an error threshold.  Section~\ref{sec:mods} presents further modifications of the two-subspace Kaczmarz method which provide even more improvements on the provable convergence bounds.  A discussion of these methods is provided in Section~\ref{sec:discuss}.  We conclude with numerical experiments demonstrating the improvements from our method in Section~\ref{sec:numerics}.

\section{Main Results} \label{sec:proofs}

We now present the proof of Theorem~\ref{thm:main}.  We first derive a bound for the expected progress made in a single iteration.  Since the two row indices are chosen independently at each iteration, we will be able to apply the bound recursively to obtain the desired overall expected convergence rate.  

Our first lemma shows that the expected estimation error in a single iteration of the two-subspace Kaczmarz method is decreased by a factor strictly less than that of the standard randomized method.

\begin{lemma}\label{lem:main}
 Let $\vct{x_{k}}$ denote the estimation to the solution of $\mtx{A}\vct{x} = \vct{b}$ in the $k$th iteration of the two-subspace Kaczmarz method.  Denote the rows of $\mtx{A}$ by $\vct{a}_1, \vct{a}_2, \ldots \vct{a}_m$.  Then we have the following bound,
$$
\mathbb{E}\|\vct{x} - \vct{x_{k}}\|_2^2 \leq \left(1 - \frac{1}{R}\right)^2\|\vct{x} - \vct{x_{k-1}}\|_2^2 - \frac{1}{m^2-m}\sum_{r < s} C_{r,s}^2\left(\langle \vct{x} - \vct{x_{k-1}}, \vct{a_{r}}\rangle^2 + \langle \vct{x} - \vct{x_{k-1}}, \vct{a_{s}}\rangle^2\right),
$$

where $C_{r,s} = \frac{|\mu_{r,s}|-\mu_{r,s}^2}{\sqrt{1-\mu_{r,s}^2}}$, $\mu_{r,s} = \<\vct{a_r}, \vct{a_s}\>$, and $R = \|\mtx{A}^{-1}\|^2\|\mtx{A}\|_F^2$ denotes the scaled condition number.

\end{lemma}

\begin{proof}
  We fix an iteration $k$ and for convenience refer to $\vct{v}_k$, $\mu_k$, and $\vct{y}_k$ as $\vct{v}$, $\mu$, and $\vct{y}$, respectively.  We will also denote $\gamma = \langle \vct{a_{r}}, \vct{v}\rangle$.  
  
 First, observe that by the definitions of $\vct{v}$ and $\vct{x_{k}}$ we have
$$
\vct{x_{k}} = \vct{x_{k-1}} + \langle \vct{x} - \vct{x_{k-1}}, \vct{a_{s}}\rangle \vct{a_{s}} + \langle \vct{x} - \vct{x_{k-1}}, \vct{v}\rangle \vct{v}.
$$
Since $\vct{a_{s}}$ and $\vct{v}$ are orthonormal, this gives the estimate
\begin{equation}\label{eq:new}
\|\vct{x} - \vct{x_{k}}\|_2^2 = \|\vct{x} - \vct{x_{k-1}}\|_2^2 - |\langle \vct{x} - \vct{x_{k-1}}, \vct{a_{s}}\rangle|^2 - |\langle \vct{x} - \vct{x_{k-1}}, \vct{v}\rangle|^2
\end{equation}

We wish to compare this error with the error from the standard randomized Kaczmarz method.  Since we utilize two rows per iteration in the two-subspace Kaczmarz method, we compare its error with the error from two iterations of the standard method.   Let $\vct{z}$ and $\vct{z'}$ be two subsequent estimates in the standard method following the estimate  $\vct{x_{k-1}}$, and assume $\vct{z} \neq \vct{z'}$.  That is,
\begin{equation}\label{eq:z}
\vct{z} = \vct{x_{k-1}} +  (b_r - \langle \vct{x_{k-1}}, \vct{a_{r}}\rangle)\vct{a_{r}} \quad\text{and}\quad
\vct{z'} = \vct{z} +  (b_s - \langle \vct{z}, \vct{a_{s}}\rangle)\vct{a_{s}}.
\end{equation} 

Recalling the definitions of $\vct{v}$, $\mu$ and $\gamma$, we have
\begin{equation}\label{eq:a}
\vct{a_{r}} = \mu\vct{a_{s}} + \gamma \vct{v}\quad\text{with}\quad \mu^2 + \gamma^2 = 1.
\end{equation}

Substituting this into~\eqref{eq:z} yields
$$
\vct{z} = \vct{x_{k-1}} + \mu\langle \vct{x} - \vct{x_{k-1}}, \vct{a_{r}}\rangle \vct{a_{s}} + \gamma\langle\vct{x} - \vct{x_{k-1}}, \vct{a_{r}}\rangle\vct{v}.
$$

Now substituting this into~\eqref{eq:z} and taking the orthogonality of $\vct{a_{s}}$ and $\vct{v}$ into account,
$$
\vct{z}' = \vct{x_{k-1}} + \langle \vct{x} - \vct{x_{k-1}}, \vct{a_{s}}\rangle \vct{a_{s}} + \gamma\langle\vct{x} - \vct{x_{k-1}}, \vct{a_{r}}\rangle\vct{v}.
$$

For convenience, let $\vct{e_{k-1}} = \vct{x} - \vct{x_{k-1}}$ denote the error in the $(k-1)$st iteration of two-subspace Kaczmarz. Then we have
\begin{align*}
\|\vct{x} - \vct{z'}\|_2^2 
&= \|\vct{e_{k-1}} - \langle \vct{e_{k-1}}, \vct{a_{s}}\rangle \vct{a_{s}} - \gamma\langle\vct{e_{k-1}}, \vct{a_{r}}\rangle\vct{v}\|_2^2\\
&= \|\vct{e_{k-1}} - \langle \vct{e_{k-1}}, \vct{a_{s}}\rangle \vct{a_{s}} - \langle \vct{e_{k-1}}, \vct{v}\rangle\vct{v} - (\gamma\langle\vct{e_{k-1}}, \vct{a_{r}}\rangle - \langle \vct{e_{k-1}}, \vct{v}\rangle)\vct{v}\|_2^2\\
&= \|\vct{e_{k-1}}\|_2^2 - |\langle \vct{e_{k-1}}, \vct{a_{s}}\rangle|^2 - |\langle \vct{e_{k-1}}, \vct{v}\rangle|^2 + |\gamma\langle\vct{e_{k-1}}, \vct{a_{r}}\rangle - \langle \vct{e_{k-1}}, \vct{v}\rangle|^2.
\end{align*}
The third equality follows from the orthonormality of $\vct{a_{s}}$ and $\vct{v}$.
We now expand the last term,
\begin{align*}
|\gamma\langle\vct{e_{k-1}}, \vct{a_r}\rangle - \langle \vct{e_{k-1}}, \vct{v}\rangle|^2 &=
|\gamma\langle\vct{e_{k-1}}, \mu\vct{a_{s}} + \gamma\vct{v}\rangle - \langle \vct{e_{k-1}}, \vct{v}\rangle|^2\\
&=  |\gamma^2\langle\vct{e_{k-1}}, \vct{v}\rangle + \gamma\mu\langle \vct{e_{k-1}}, \vct{a_{s}} \rangle- \langle \vct{e_{k-1}}, \vct{v}\rangle|^2\\
&= |\mu^2\langle\vct{e_{k-1}}, \vct{v}\rangle - \gamma\mu\langle \vct{e_{k-1}}, \vct{a_{s}} \rangle|^2.
\end{align*}
This gives
\begin{align*}
\|\vct{x} - \vct{z'}\|_2^2 
&= \|\vct{e_{k-1}}\|_2^2 - |\langle \vct{e_{k-1}}, \vct{a_{s}}\rangle|^2 - |\langle \vct{e_{k-1}}, \vct{v}\rangle|^2 +|\mu^2\langle\vct{e_{k-1}}, \vct{v}\rangle - \gamma\mu\langle \vct{e_{k-1}}, \vct{a_{s}} \rangle|^2.
\end{align*}

Combining this identity with~\eqref{eq:new}, we now relate the expected error in the two-subspace Kaczmarz algorithm, $\mathbb{E}\|\vct{x} - \vct{x_{k}}\|_2^2$ to the expected error of the standard method, $\mathbb{E}\|\vct{x} - \vct{z'}\|_2^2$ as follows:
\begin{equation}\label{eq:stand}
\mathbb{E}\|\vct{x} - \vct{x_{k}}\|_2^2= \mathbb{E}\|\vct{x} - \vct{z'}\|_2^2 - \mathbb{E}|\mu^2\langle\vct{e_{k-1}}, \vct{v}\rangle - \gamma\mu\langle \vct{e_{k-1}}, \vct{a_{s}} \rangle|^2.
\end{equation}

It thus remains to analyze the last term. Since we select the two rows $r$ and $s$ independently from the uniform distribution over pairs of distinct rows, the expected error is just the average of the error over all $m^2 - m$ ordered choices ${r,s}$.  To this end we introduce the notation $\mu_{r,s} = \langle \vct{a_{r}}, \vct{a_{s}}\rangle$.  Then by definitions of $\vct{v}$, $\mu$ and $\gamma$,  
\begin{align*}
\mathbb{E}|\mu^2\langle\vct{e_{k-1}}&, \vct{v}\rangle - \gamma\mu\langle \vct{e_{k-1}}, \vct{a_{s}} \rangle|^2\\
&= \frac{1}{m^2-m}\sum_{r\ne s} \left|\frac{\mu_{r,s}^2}{\sqrt{1-\mu_{r,s}^2}} (\langle \vct{e_{k-1}}, \vct{a_{r}}\rangle - \mu_{r,s}\langle \vct{e_{k-1}}, \vct{a_{s}}\rangle) - \mu_{r,s}\sqrt{1 - \mu_{r,s}^2}\langle \vct{e_{k-1}}, \vct{a_{s}}\rangle  \right|^2\\
&= \frac{1}{m^2-m}\sum_{r\ne s} \left|\frac{\mu_{r,s}^2}{\sqrt{1-\mu_{r,s}^2}} \langle \vct{e_{k-1}}, \vct{a_{r}}\rangle - \left(\frac{\mu_{r,s}^3}{\sqrt{1-\mu_{r,s}^2}} + \mu_{r,s}\sqrt{1 - \mu_{r,s}^2} \right)\langle \vct{e_{k-1}}, \vct{a_{s}}\rangle\right|^2\\
&= \frac{1}{m^2-m}\sum_{r\ne s} \left|\frac{\mu_{r,s}^2}{\sqrt{1-\mu_{r,s}^2}} \langle \vct{e_{k-1}}, \vct{a_{r}}\rangle - \left(\frac{\mu_{r,s}}{\sqrt{1-\mu_{r,s}^2}} \right)\langle \vct{e_{k-1}}, \vct{a_{s}}\rangle\right|^2.\\
\end{align*}


We now recall that for any $\theta,\pi,u,$ and $v$,
$$
(\theta u - \pi v)^2 + (\theta v - \pi u)^2 \geq (|\pi| - |\theta|)^2(u^2+v^2). 
$$
  
  Setting $\theta_{r,s} = \frac{\mu_{r,s}^2}{\sqrt{1-\mu_{r,s}^2}}$ and $\pi_{r,s} = \frac{\mu_{r,s}}{\sqrt{1-\mu_{r,s}^2}} $, we have by rearranging terms in the symmetric sum,

\begin{align}\label{eq:bound}
\mathbb{E}|\mu^2\langle\vct{e_{k-1}}&, \vct{\theta}\rangle - \gamma\mu\langle \vct{e_{k-1}}, \vct{a_{s}} \rangle|^2\notag\\
&= \frac{1}{m^2-m}\sum_{r\ne s} \left|\theta_{r,s} \langle \vct{e_{k-1}}, \vct{a_{r}}\rangle - \pi_{r,s}\langle \vct{e_{k-1}}, \vct{a_{s}}\rangle\right|^2\notag\\
&= \frac{1}{m^2-m}\sum_{r < s} \left|\theta_{r,s} \langle \vct{e_{k-1}}, \vct{a_{r}}\rangle - \pi_{r,s}\langle \vct{e_{k-1}}, \vct{a_{s}}\rangle\right|^2 + \left|\theta_{r,s} \langle \vct{e_{k-1}}, \vct{a_{s}}\rangle - \pi_{r,s}\langle \vct{e_{k-1}}, \vct{a_{r}}\rangle\right|^2\notag\\
&\geq \frac{1}{m^2-m}\sum_{r < s} (|\pi_{r,s}| - |\theta_{r,s}|)^2\left((\langle \vct{e_{k-1}}, \vct{a_{r}}\rangle)^2 + (\langle \vct{e_{k-1}}, \vct{a_{s}}\rangle)^2\right) \notag\\
&= \frac{1}{m^2-m}\sum_{r < s} \Big( \frac{|\mu_{r,s}|-\mu_{r,s}^2}{\sqrt{1-\mu_{r,s}^2}} \Big)^2\left(\langle \vct{e_{k-1}}, \vct{a_{r}} \rangle)^2 + (\langle \vct{e_{k-1}}, \vct{a_{s}}\rangle)^2\right).
\end{align}

Since selecting two rows without replacement (i.e. guaranteeing not to select the same row back to back) can only speed the convergence, we have from~\eqref{SV} that the error from the standard randomized Kaczmarz method satisfies
$$
\E\|\vct{x}-\vct{z}'\|_2^2 \leq (1 - 1/R)^2\|\vct{x} - \vct{x_{k-1}}\|_2^2.
$$

Combining this with~\eqref{eq:stand} and~\eqref{eq:bound} yields the desired result.

\end{proof}

Although the result of Lemma~\ref{lem:main} is tighter, using the coherence parameters $\delta$ and $\Delta$ of~\eqref{mus} allows us to present the following looser but simpler result. 

\begin{lemma}\label{lem:simple}
 Let $\vct{x_{k}}$ denote the estimation to $\mtx{A}\vct{x} = \vct{b}$ in the $k$th iteration of the two-subspace Kaczmarz method.  Denote the rows of $\mtx{A}$ by $\vct{a}_1, \vct{a}_2, \ldots \vct{a}_m$.  Then 
$$
\mathbb{E}\|\vct{x} - \vct{x_{k}}\|_2^2 \leq \left(\left(1 - \frac{1}{R}\right)^2 - \frac{D}{R} \right)\|\vct{x} - \vct{x_{k-1}}\|_2^2 ,
$$

where $D = \min\Big\{ \frac{\delta^2(1-\delta)}{1+\delta}, \frac{\Delta^2(1-\Delta)}{1+\Delta} \Big\}$, $\delta$ and $\Delta$ are the coherence parameters as in~\eqref{mus}, and $R = \|\mtx{A}^{-1}\|^2\|\mtx{A}\|_F^2$ denotes the scaled condition number.
\end{lemma}
\begin{proof}
By Lemma~\ref{lem:main} we have
\begin{equation}\label{eq:lemma}
\mathbb{E}\|\vct{x} - \vct{x_{k}}\|_2^2 \leq \left(1 - \frac{1}{R}\right)^2\|\vct{x} - \vct{x_{k-1}}\|_2^2 - \frac{1}{m^2-m}\sum_{r < s} C_{r,s}^2\left(\langle \vct{x} - \vct{x_{k-1}}, \vct{a_{r}}\rangle^2 + \langle \vct{x} - \vct{x_{k-1}}, \vct{a_{s}}\rangle^2\right),
\end{equation}
where 
$$
C_{r,s} = \frac{|\langle\vct{a_{r}}, \vct{a_{s}}\rangle|-\langle\vct{a_{r}}, \vct{a_{s}}\rangle^2}{\sqrt{1-\langle\vct{a_{r}}, \vct{a_{s}}\rangle^2}}.
$$

By the assumption that $\delta \leq |\langle \vct{a_{r}}, \vct{a_{s}} \rangle| \leq \Delta$, we have 
$$
C_{r,s}^2 \geq \min\Big\{ \frac{\delta^2(1-\delta)}{1+\delta}, \frac{\Delta^2(1-\Delta)}{1+\Delta} \Big\} = D.
$$

Thus we have that 
\begin{align}\label{eq:final}
\frac{1}{m^2-m}\sum_{r < s} &C_{r,s}^2\left(\langle \vct{x} - \vct{x_{k-1}}, \vct{a_{r}}\rangle^2 + \langle \vct{x} - \vct{x_{k-1}}, \vct{a_{s}}\rangle^2\right)\notag\\
&\geq \frac{D}{m^2-m}\sum_{r < s} \left(\langle \vct{x} - \vct{x_{k-1}}, \vct{a_{r}}\rangle^2 + \langle \vct{x} - \vct{x_{k-1}}, \vct{a_{s}}\rangle^2\right)\notag\\
&= \frac{D(m-1)}{m^2-m}\sum_{r=1}^m \langle \vct{x} - \vct{x_{k-1}}, \vct{a_{r}}\rangle^2\notag\\
&\geq \frac{D}{m}\cdot\frac{\|\vct{x} - \vct{x_{k-1}}\|_2^2}{\|\mtx{A}^{-1}\|_2^2}.
\end{align}

In the last inequality we have employed the fact that for any $\vct{z}$,
$$
\sum_{r=1}^m \langle \vct{z}, \vct{a_{r}}\rangle^2 \geq \frac{\|\vct{z}\|_2^2}{\|\mtx{A}^{-1}\|_2^2}.
$$

Combining~\eqref{eq:final} and~\eqref{eq:lemma} along with the definition of $R$ yields the claim.

\end{proof}

Applying Lemma~\ref{lem:simple} recursively and using the fact that the selection of rows in each iteration is independent yields our main result Theorem~\ref{thm:main}.

\section{Noisy Systems}\label{sec:noise}

Next we consider systems which have been perturbed by noise.  The inconsistent system $\vct{b} = \mtx{A}\vct{x}$ now becomes (the possibly inconsistent system) $\vct{b} = \mtx{A}\vct{x} + \vct{w}$ for some error vector $\vct{w}$.  As evident from~\eqref{eq:noise}, the standard method with noise exhibits exponential convergence down to an error threshold, which is proportional to $\|\vct{w}\|_{\infty}$.  Our main result in the noisy case is that the two-subspace version again exhibits even faster exponential convergence, down to a threshold also proportional to $\|\vct{w}\|_{\infty}$.

\begin{theorem}\label{thm:noise}
Let $\mtx{A}$ be a full rank matrix with $m$ rows and suppose $\vct{b} = \mtx{A}\vct{x} + \vct{w}$ is a noisy system of equations.  Let $\vct{x_k}$ denote the estimation to the solution $\vct{x}$ in the $k$th iteration of the two-subspace Kaczmarz method.  Then
$$
\mathbb{E}\|\vct{x} - \vct{x_k}\|_2 \leq \eta^{k/2}\|\vct{x} - \vct{x}_{0}\|_2 + \frac{3}{1-\sqrt{\eta}}\cdot\frac{\|\vct{w}\|_{\infty}}{\sqrt{1-\Delta^2}} ,
$$

where $\eta = \left(1 - \frac{1}{R}\right)^2 - \frac{D}{R}$, $D = \min\Big\{ \frac{\delta^2(1-\delta)}{1+\delta}, \frac{\Delta^2(1-\Delta)}{1+\Delta} \Big\}$, $\Delta$ and $\delta$ are the coherence parameters~\eqref{mus}, and $R = \|\mtx{A}^{-1}\|^2\|\mtx{A}\|_F^2$ denotes the scaled condition number. 
\end{theorem}

As in the case of our main result Theorem~\ref{thm:main}, this bound is not tight.  The same improvements mentioned in the remarks about Theorem~\ref{thm:main} can also be applied here.  In particular, the dependence on $\Delta$ seems to be only an artifact of the proof (see Section~\ref{sec:numerics}.  Nonetheless, this result still shows that the two-subspace Kaczmarz method provides expected exponential convergence down to an error threshold which is analagous to that of the standard method.  The convergence factors are again substantially better than the standard method for coherent systems.

\begin{proof}[Proof of Theorem~\ref{thm:noise}]

Fix an iteration $k$ and denote by $\vct{y}$, $\vct{v}$, $\mu$ and $\beta$ the values of $\vct{y_k}$, $\vct{v_k}$, $\mu_k$ and $\beta_k$ for convenience.  Let $\vct{y'}$ and $\beta'$ be the values of $\vct{y_k}$, and $\beta_k$ as if there were noise (i.e. $\vct{w} = 0$).  In other words, we have
$$
\vct{y} = \vct{y'} + w_r\vct{a_r} \quad\text{and}\quad \beta = \beta' + \frac{w_r + \mu w_s}{\sqrt{1-\mu^2}}.
$$

Then by the definition of $\vct{x_{k}}$, we have
$$
\vct{x_{k}} = \vct{x^*_{k}} + w_r\vct{a_r} + \frac{w_r + \mu w_s}{\sqrt{1-\mu^2}}\vct{v},
$$
where $\vct{x^*_{k}} = \vct{y'} + (\beta' - \<\vct{x_{k-1}}, \vct{v}\>)\vct{v}$ denotes the next estimation from $\vct{x_{k-1}}$ if there were no noise.  Therefore, we have that
\begin{align*}
\|\vct{x} - \vct{x_k}\|_2 &\leq \|\vct{x} - \vct{x^*_k}\|_2 + \|w_r\vct{a_r} + \frac{w_r + \mu w_s}{\sqrt{1-\mu^2}}\vct{v}\|_2\\
&\leq \|\vct{x} - \vct{x^*_k}\|_2 + |w_r|\|\vct{a_r}\|_2 + \left|\frac{w_r + \mu w_s}{\sqrt{1-\mu^2}}\right| \|\vct{v}\|_2\\
&= \|\vct{x} - \vct{x^*_k}\|_2 + |w_r| + \left|\frac{w_r + \mu w_s}{\sqrt{1-\mu^2}}\right| \\
&\leq \|\vct{x} - \vct{x^*_k}\|_2 + \frac{3\|\vct{w}\|_{\infty}}{\sqrt{1-\Delta^2}}.
\end{align*}

By Jensen's inequality and Lemma~\ref{lem:main},
$$
\E\|\vct{x} - \vct{x^*_k}\|_2 \leq \sqrt{\eta}\E\|\vct{x} - \vct{x_{k-1}}\|_2.
$$

Combining the above recursively yields
\begin{align*}
\E\|\vct{x} - \vct{x_k}\|_2 &\leq \eta^{k/2}\|\vct{x} - \vct{x_0}\|_2 + \frac{3\|\vct{w}\|_{\infty}}{\sqrt{1-\Delta^2}}\sum_{j=1}^{k-1}\eta^{j/2}\\
&\leq \eta^{k/2}\|\vct{x} - \vct{x}_{0}\|_2 + \frac{3}{1-\sqrt{\eta}}\cdot\frac{\|\vct{w}\|_{\infty}}{\sqrt{1-\Delta^2}},
\end{align*}

which proves the claim.
\end{proof}

\section{Further improvements}\label{sec:mods}

Next we state and prove a lemma which demonstrates even more improvements on the convergence rate from the standard method in the case where the correlations between the rows are non-negative.  If this is not the case, we may alter one step of the two-subspace method to generalize the result to matrices with arbitrary correlations.  This modification will decrease the factor yet again in the exponential convergence rate of the two-subspace method.  We consider the noiseless case here, although results analagous to those in Section~\ref{sec:noise} can easily be obtained using the same methods. 

We first define an $m^2\times n$ matrix $\mtx{\Omega}$ whose rows $\omega$ are differnces of the rows of $\mtx{A}$:
\begin{equation}\label{omega}
 \vct{\omega}_{m(j-1)+i} = \left\{ \begin{array}{cc} \frac{\vct{a_j} - \vct{a_i}}{\| \vct{a_j} - \vct{a_i} \|}, \quad & j,i = 1, ..., m,\quad  j \neq i, \\ 
 \vct{0}, & j=i\end{array}\right.
 \end{equation}
 We may now state our main lemma.
 
 \begin{lemma}\label{lem:improve}
 Let $\vct{x_k}$ denote the estimation to $\mtx{A}\vct{x} = \vct{b}$ in the $k$th iteration of the two-subspace Kaczmarz method.  For indices $r$ and $s$, set $\mu_{r,s} = \langle\vct{a_r}, \vct{a_s}\rangle$.  We have the following bound,
\begin{align}
\mathbb{E}\|\vct{x} - \vct{x_k}\|_2^2 &\leq \left(1 - \frac{1}{R} \right)^2\|\vct{x} - \vct{x_{k-1}}\|_2^2 - \frac{1}{m^2-m}\sum_{r < s}  C_{r,s} \left(\langle \vct{x} - \vct{x_{k-1}}, \vct{a_r}\rangle^2 + \langle \vct{x} - \vct{x_{k-1}}, \vct{a_s}\rangle^2\right) \notag \\
&- \frac{1}{m^2-m} \sum_{j,i=1}^m E_{i,j} \langle \vct{x} - \vct{x}_{k-1}, \vct{\omega_{m(j-1)+i}} \rangle^2 \end{align}

where $C_{r,s} = \frac{\mu_{r,s}^2(1-\mu_{r,s})}{1+\mu_{r,s}}$ and $E_{i,j} = 4 \mu_{i,j}^3$.

\end{lemma}
 
\begin{remark} If the correlations $\mu_{r,s} = \langle\vct{a_r}, \vct{a_s}\rangle$ between the rows of $\mtx{A}$ are non-negative, then the constants $E_{i,j}$ are all non-negative and thus this lemma offers a strict improvement over Lemma~\ref{lem:main}.  However, if this is not the case, this bound may actually be worse than that of Lemma~\ref{lem:main}.  To overcome this, we may simply modify the two-subspace Kaczmarz algorithm so that in each iteration $\mu_{r,s}$ is non-negative (by possibly using $-\vct{a_r}$ instead of $\vct{a_r}$ when needed for example).  This modification seems necessary only for the proof, and empirical results for the modified and unmodified methods remain the same.  From this point on, we will assume this modification is in place.
\end{remark}

\begin{proof}[Proof of Lemma~\ref{lem:improve}]
  We again fix an iteration $k$ and for convenience refer to $\vct{v}_k$, $\mu_k$, and $\vct{y}_k$ as $\vct{v}$, $\mu$, and $\vct{y}$, respectively.  We will also let $\vct{e_{k-1}} = \vct{x} - \vct{x_{k-1}}$ be the error in the $(k-1)$st iteration. 
  
By~\eqref{eq:stand} and~\eqref{SV}, we have that 
\begin{equation}\label{eq:restart}
\mathbb{E}\|\vct{x} - \vct{x_k}\|_2^2 = (1 - 1/R)\|\vct{x} - \vct{x_{k-1}}\|_2^2 + \mathbb{E}|\mu^2\langle\vct{e_{k-1}}, \vct{v}\rangle - \gamma\mu\langle \vct{e_{k-1}}, \vct{a_{s}} \rangle|^2.
\end{equation}

We now analyze the last term carefully. To take expectation we must look over all combinations of choices ${r,s}$, so to that end denote $\langle \vct{a_r}, \vct{a_s}\rangle$ by $\mu_{r,s}$.  Since we select two rows uniformly at random (with replacement), using the definitions of $\vct{c}$, $\mu$ and $\gamma$, we have

\begin{align*}
\mathbb{E}|\mu^2\langle\vct{e}_{k-1}&, \vct{c}\rangle - \gamma\mu\langle \vct{e}_{k-1}, \vct{a_s} \rangle|^2\\
&= \frac{1}{m^2-m}\sum_{r\ne s} \left|\frac{\mu_{r,s}^2}{\sqrt{1-\mu_{r,s}^2}} (\langle \vct{e}_{k-1}, \vct{a_r}\rangle - \mu_{r,s}\langle \vct{e}_{k-1}, \vct{a_s}\rangle) - \mu_{r,s}\sqrt{1 - \mu_{r,s}^2}\langle \vct{e}_{k-1}, \vct{a_s}\rangle  \right|^2\\
&= \frac{1}{m^2-m}\sum_{r\ne s} \left|\frac{\mu_{r,s}^2}{\sqrt{1-\mu_{r,s}^2}} \langle \vct{e}_{k-1}, \vct{a_r}\rangle - \left(\frac{\mu_{r,s}^3}{\sqrt{1-\mu_{r,s}^2}} + \mu_{r,s}\sqrt{1 - \mu_{r,s}^2} \right)\langle \vct{e}_{k-1}, \vct{a_s}\rangle\right|^2 \\
&= \frac{1}{m^2-m}\sum_{r\ne s} \frac{\mu_{r,s}^2}{1-\mu_{r,s}^2} \Big( \mu_{r,s} \langle \vct{e}_{k-1}, \vct{a_r}\rangle - \langle \vct{e}_{k-1}, \vct{a_s}\rangle \Big)^2
\end{align*}

Now observe that 
\begin{eqnarray}
\label{eq:alphabeta}
&& \big( \mu_{r,s}  \langle \vct{e}_{k-1}, \vct{a_r}\rangle -  \langle \vct{e}_{k-1}, \vct{a_s}\rangle \big)^2 + \big( \mu_{r,s}  \langle \vct{e}_{k-1}, \vct{a_s}\rangle -  \langle \vct{e}_{k-1}, \vct{a_r}\rangle \big)^2 \nonumber \\
&=& (1-\mu_{r,s})^2 \Big(  \langle \vct{e}_{k-1}, \vct{a_r}\rangle^2 +  \langle \vct{e}_{k-1}, \vct{a_s}\rangle^2 \Big) + 2 \mu_{r,s} \Big(  \langle \vct{e}_{k-1}, \vct{a_r} - \vct{a_s} \rangle \Big)^2 \nonumber \\
&=& (1-\mu_{r,s})^2 \Big(  \langle \vct{e}_{k-1}, \vct{a_r}\rangle^2 +  \langle \vct{e}_{k-1}, \vct{a_s}\rangle^2 \Big) + 4 \mu_{r,s} (1 - \mu_{r,s}^2) \Big(  \langle \vct{e}_{k-1}, \frac{\vct{a_r} - \vct{a_s}}{ \| \vct{a_r} - \vct{a_s} \| } \rangle \Big)^2 \nonumber 
\end{eqnarray}

  Thus taking advantage of the symmetry in the sum we have,

\begin{align}\label{eq:boundSym}
\mathbb{E}|\mu^2\langle\vct{e}_{k-1}&, \vct{c}\rangle - \gamma\mu\langle \vct{e}_{k-1}, \vct{a_s} \rangle|^2\notag\\
&= \frac{1}{m^2-m}\sum_{r < s} \frac{\mu_{r,s}^2 (1-\mu_{r,s})}{1+\mu_{r,s}} \Big(  \langle \vct{e}_{k-1}, \vct{a_r}\rangle^2 +  \langle \vct{e}_{k-1}, \vct{a_s}\rangle^2 \Big)  + 4 \mu_{r,s}^3  \Big(  \langle \vct{e}_{k-1}, \frac{\vct{a_r} - \vct{a_s}}{ \| \vct{a_r} - \vct{a_s} \| } \rangle \Big)^2   \notag\\
\end{align}

Combining with~\eqref{eq:restart} yields the claim.

\end{proof}
We may now use the coherence parameters $\delta$ and $\Delta$ from~\eqref{mus} to obtain the following simplified result.

\begin{lemma}\label{lem:improvesimple}
 Let $\vct{x_k}$ denote the estimation to $\mtx{A}\vct{x} = \vct{b}$ in the $k$th iteration of the two-subspace Kaczmarz method.  Then,
$$
\mathbb{E}\|\vct{x} - \vct{x_k}\|_2^2 \leq \left(\left(1 - \frac{1}{R}\right)^2 - \frac{D}{R} - \frac{E}{Q} \right)\|\vct{x} - \vct{x_{k-1}}\|_2^2 ,
$$

where $D = \min\Big\{ \frac{\delta^2(1-\delta)}{1+\delta}, \frac{\Delta^2(1-\Delta)}{1+\Delta} \Big\}$, $E = 4\delta^3$ and $R = \|\mtx{A}^{-1}\|^2\|\mtx{A}\|_F^2$ and $Q =  \|\mtx{\Omega}^{-1}\|^2\|\mtx{\Omega}\|_F^2$ denote the scaled condition numbers of $\mtx{A}$ and $\mtx{\Omega}$ (from~\eqref{omega}, respectively.
\end{lemma}

\begin{proof}

In light of Lemma~\ref{lem:improve} and the proof of Lemma~\ref{lem:simple}, it suffices to show that
$$
\frac{1}{m^2-m}\sum_{j,i=1}^m E_{i,j} \langle \vct{x} - \vct{x}_{k-1}, \vct{\omega_{m(j-1)+i}} \rangle^2 \geq \frac{E}{Q}\|\vct{x} - \vct{x}_{k-1}\|_2^2.
$$

By the definition~\eqref{mus} of $\delta$ and $\|\Omega^{-1}\|$, we have
\begin{align*}
\frac{1}{m^2-m}\sum_{i\ne j}^m E_{i,j} \langle \vct{x} - \vct{x}_{k-1}, \vct{\omega_{m(j-1)+i}} \rangle^2 &\geq \frac{4\delta^3}{m^2-m}\sum_{i \ne j}^m  \langle \vct{x} - \vct{x}_{k-1}, \vct{\omega_{m(j-1)+i}} \rangle^2\\
&\geq \frac{4\delta^3}{m^2-m}\cdot\frac{\|\vct{x} - \vct{x}_{k-1}\|_2^2}{\|\Omega^{-1}\|}\\
&\geq \frac{E}{Q}.
\end{align*}

The last equality follows since the rows of $\Omega$ are unit norm and equal to zero for $i=j$.

\end{proof}

Applying Lemma~\ref{lem:improvesimple} recursively yields our main theorem.

\begin{theorem}\label{thm:improve}
 Let $\vct{x_k}$ denote the estimation to $\mtx{A}\vct{x} = \vct{b}$ in the $k$th iteration of the two-subspace Kaczmarz method.  Then,
$$
\mathbb{E}\|\vct{x} - \vct{x_k}\|_2^2 \leq \left(\left(1 - \frac{1}{R}\right)^2 - \frac{D}{R} -  \frac{E}{Q} \right)^k\|\vct{x} - \vct{x_{0}}\|_2^2 ,
$$

where $D = \min\Big\{ \frac{\delta^2(1-\delta)}{1+\delta}, \frac{\Delta^2(1-\Delta)}{1+\Delta} \Big\}$, $E = 4\delta^3$ and $R = \|\mtx{A}^{-1}\|^2\|\mtx{A}\|_F^2$ and $Q =  \|\mtx{\Omega}^{-1}\|^2\|\mtx{\Omega}\|_F^2$ denote the scaled condition numbers of $\mtx{A}$ and $\mtx{\Omega}$ (from~\eqref{omega}, respectively.
\end{theorem}

\section{Numerical Results}\label{sec:numerics}

Next we perform several experiments to compare the convergence rate of the two-subspace randomized Kaczmarz with that of the standard randomized Kaczmarz method.  As discussed, both methods exhibit exponential convergence in expectation, but in many regimes the constant factor in the exponential bound of the two-subspace method is much smaller, yielding much faster convergence. 

To test these methods, we construct various types of $500\times 50$ matrices $\mtx{A}$.  To get a range of $\delta$ and $\Delta$, we set the entries of $\mtx{A}$ to be independent indentically distributed uniform random variables on some interval $[c, 1]$.  Changing the value of $c$ will appropriately change the values of $\delta$ and $\Delta$.  Note that there is nothing special about this interval, other intervals (both negative and positive or both) of varying widths yield the same results.  For each matrix construction, both the randomized Kaczmarz and two-subspace randomized methods are run with the same initial (randomly selected) estimate.  The estimation errors are computed at each iteration.  Since each iteration of the two-subspace method utilizes two rows of the matrix $\mtx{A}$, we call a single iteration of the standard method two iterations in the Algorithm~\ref{alg:rk} for fair comparison.

Figure~\ref{fig:highCoh} demonstrates the regime where the two-subspace method offers the most improvement over the standard method.  Here the matrix $A$ has highly coherent rows, with $\delta \approx \Delta \approx 1$. 

  \begin{figure}[h!]
\begin{center}
\includegraphics[width=3in]{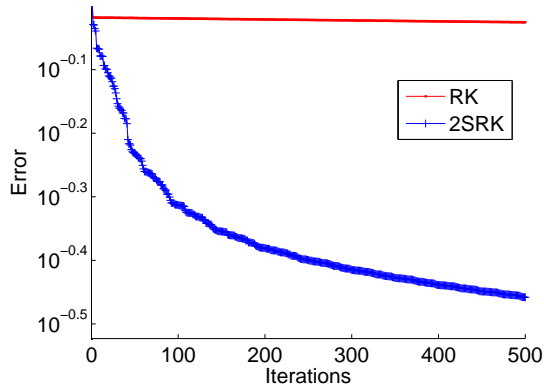}    
\end{center}
\caption{A log-linear plot of the error per iteration for the randomized Kaczmarz (RK) and two-subspace RK (2SRK).  Matrix $\mtx{A}$ has uniformly distributed highly coherent rows with $\delta = 0.992$ and $\Delta = 0.998$. }\label{fig:highCoh}
\end{figure}

Our result Theorem~\ref{thm:main} suggests that as $\delta$ becomes smaller the two-subspace method should offer less and less improvements over the standard method.  When $\delta=0$ the convergence rate bound of Theorem~\ref{thm:main} is precisely the same as that of the standard method~\eqref{SV}.  Indeed, we see this precise behavior as is depicted in Figure~\ref{fig:others}.  

\begin{figure}[h!]
\begin{center}
$\begin{array}{c@{\hspace{.1in}}c}
(a) \includegraphics[width=2.5in]{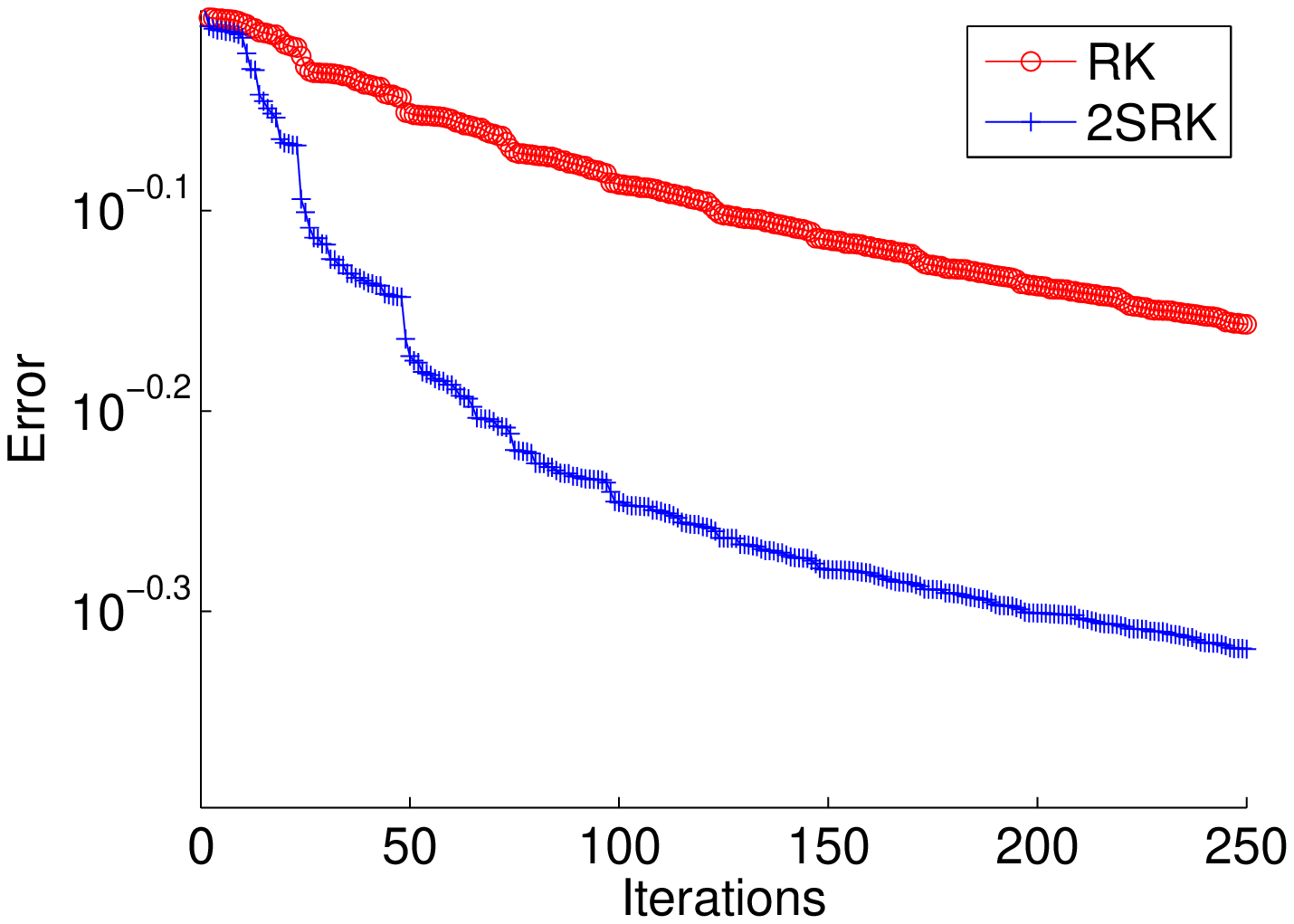} &
(b) \includegraphics[width=2.5in]{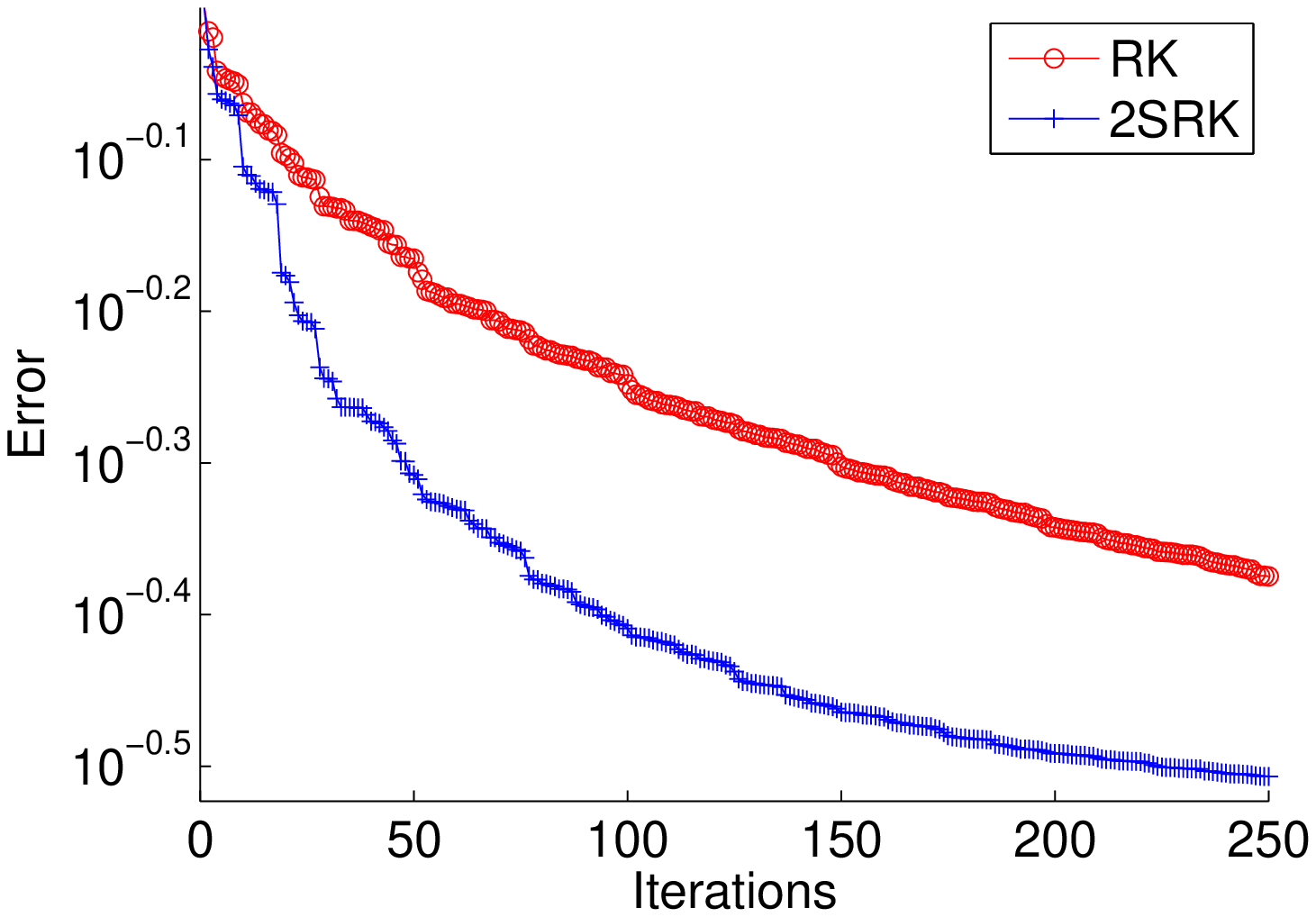} \\
(c) \includegraphics[width=2.5in]{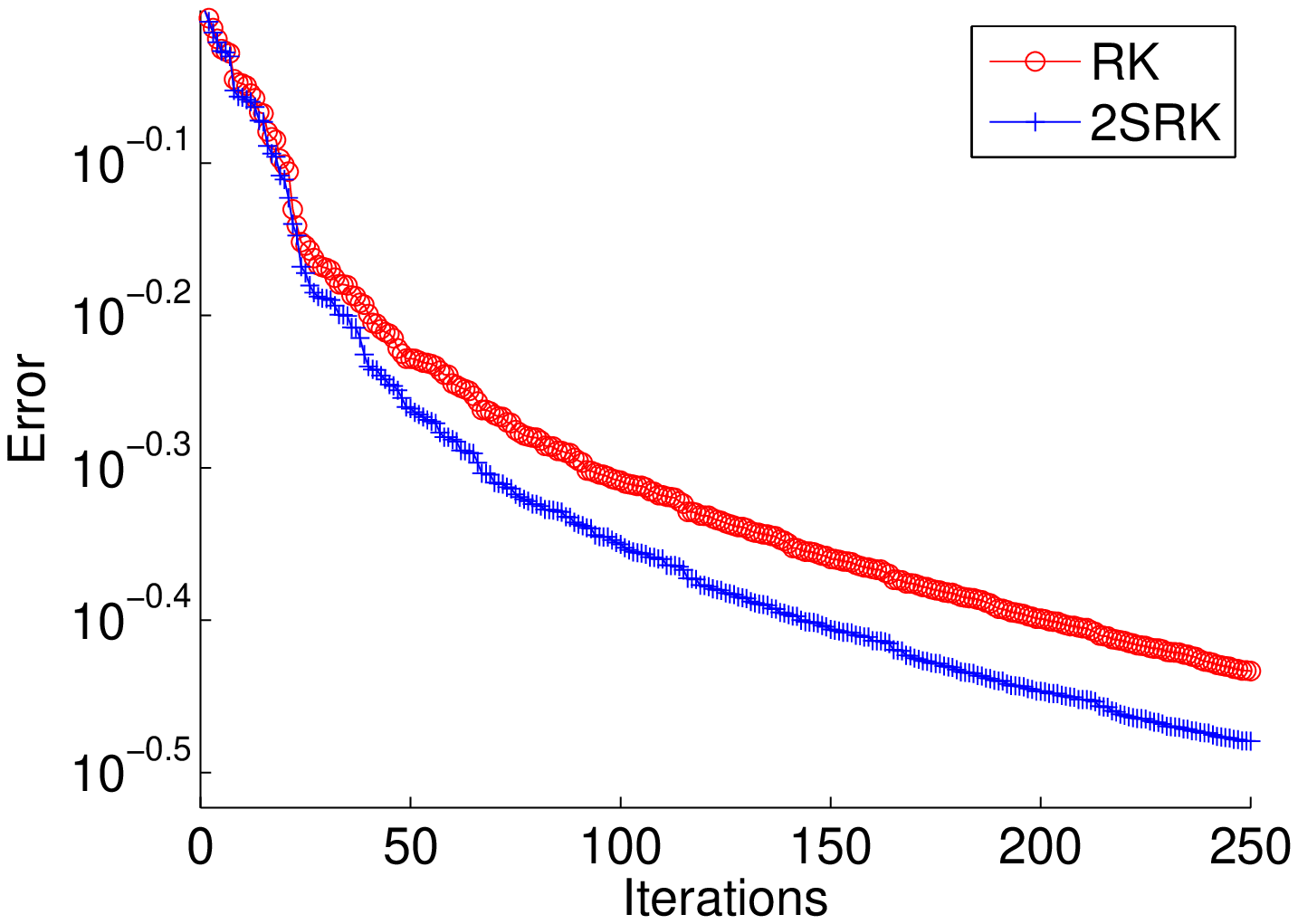} &
(d) \includegraphics[width=2.5in]{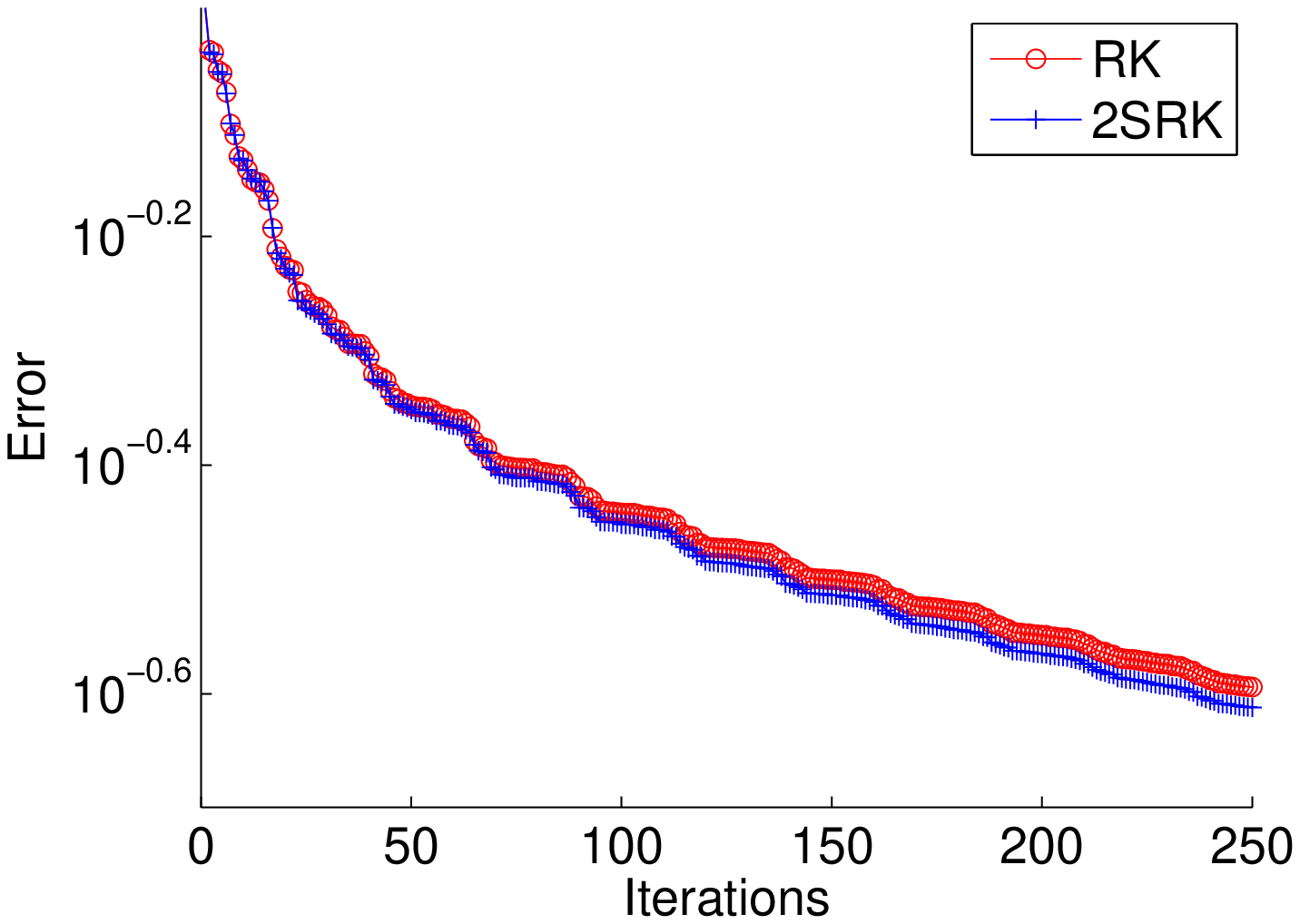}\\
\end{array}$    
\end{center}
\caption{A log-linear plot of the error per iteration for the randomized Kaczmarz (RK) and two-subspace RK (2SRK).  Matrix $\mtx{A}$ has uniformly distributed highly coherent rows with (a) $\delta = 0.837$ and $\Delta = 0.967$, (b) $\delta = 0.534$ and $\Delta = 0.904$, (c) $\delta = 0.018$ and $\Delta = 0.819$, and (d) $\delta = 0$ and $\Delta = 0.610$. }\label{fig:others}
\end{figure}

Next we performed experiments on noisy systems.  We used the same dimensions and construction of the matrix $\mtx{A}$ as well as the signal type.  Then we added i.i.d. Gaussian noise with norm $0.1$ to the measurements $\vct{b}$.  Figure~\ref{fig:noise1} demonstrates the exponential convergence of the methods in the presence of noise for various values of $\delta$ and $\Delta$.

\begin{figure}[h!]
\begin{center}
$\begin{array}{c@{\hspace{.1in}}c}
(a) \includegraphics[width=2.5in]{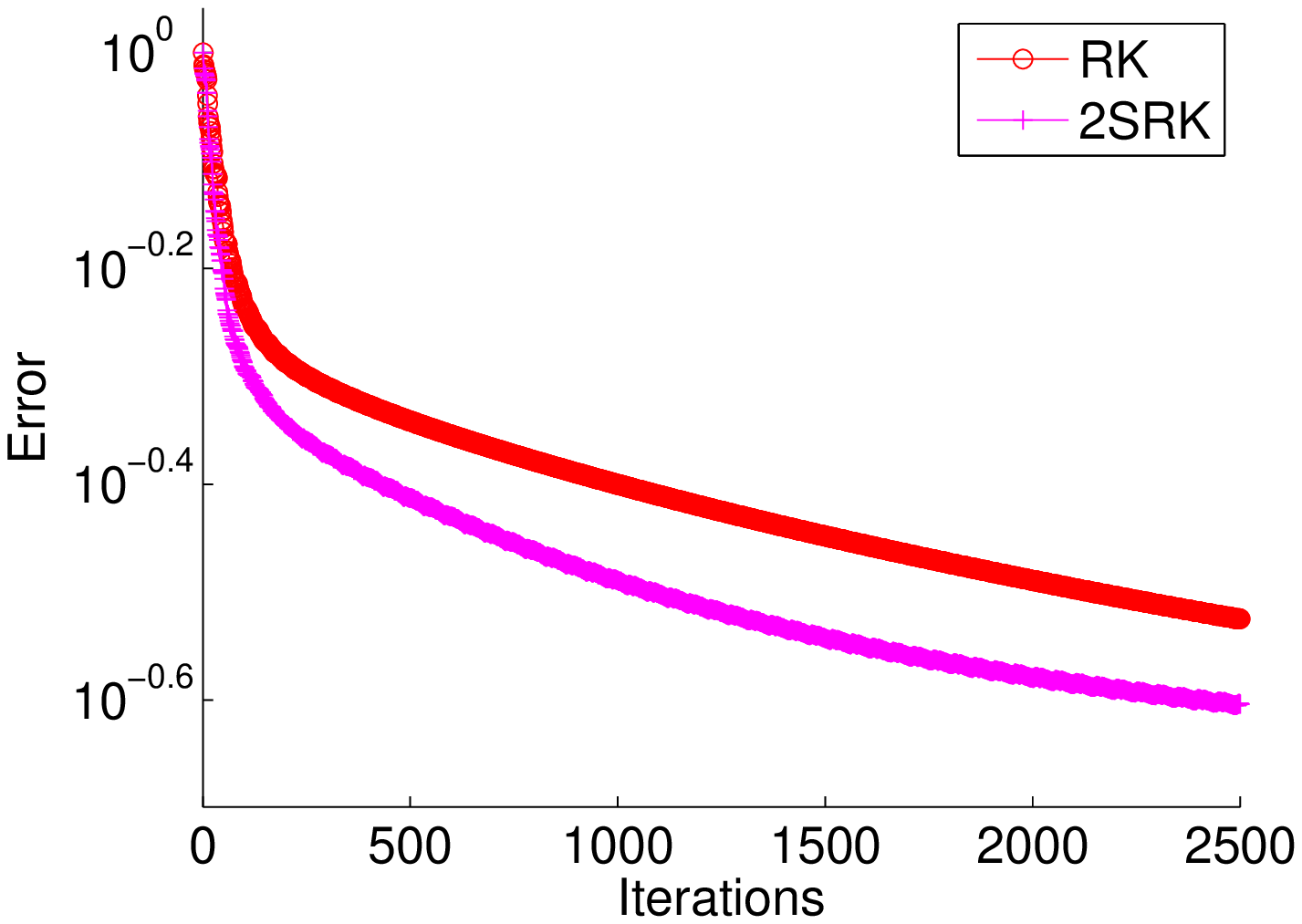} &
(b) \includegraphics[width=2.5in]{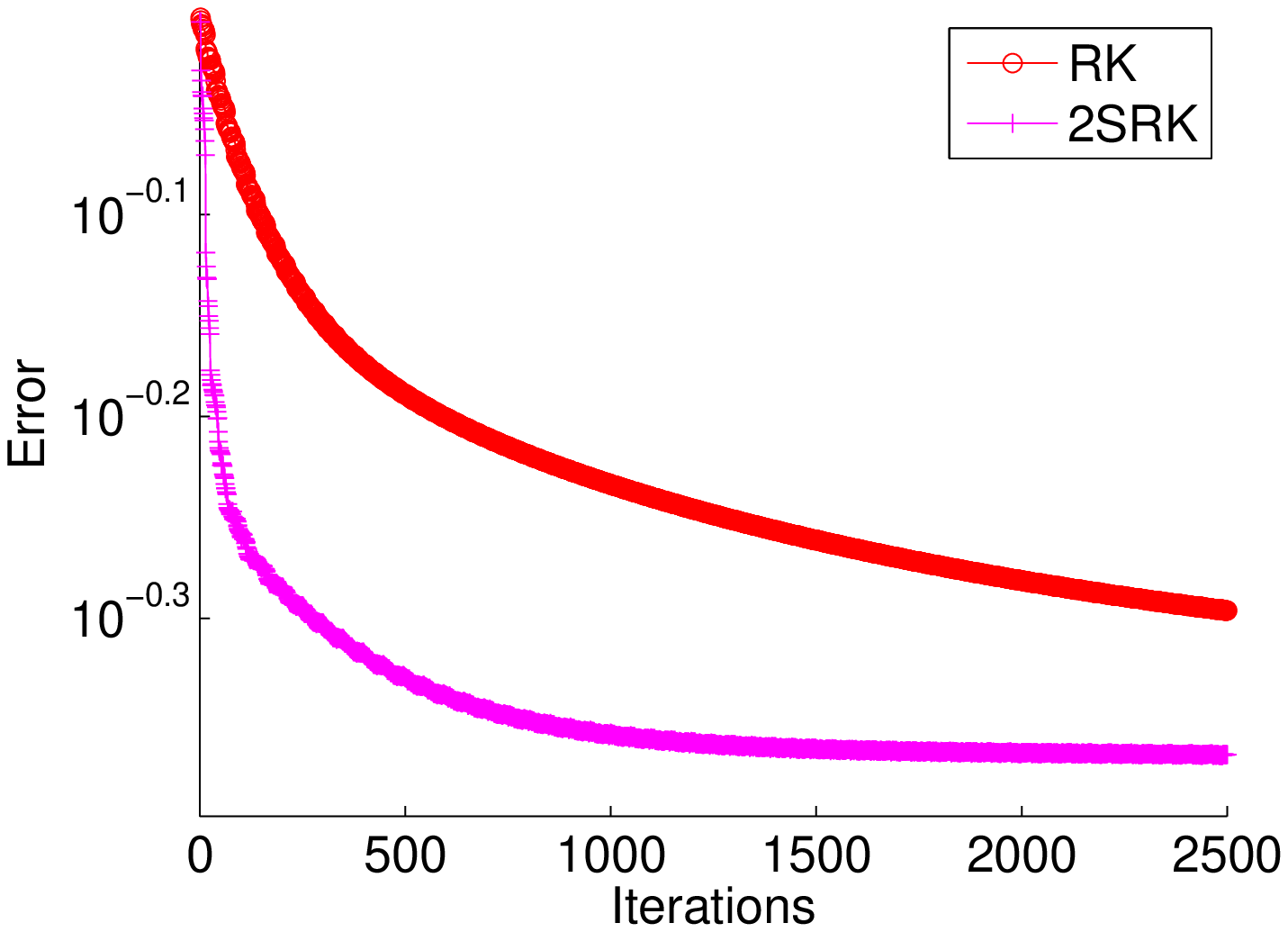} \\
(c) \includegraphics[width=2.5in]{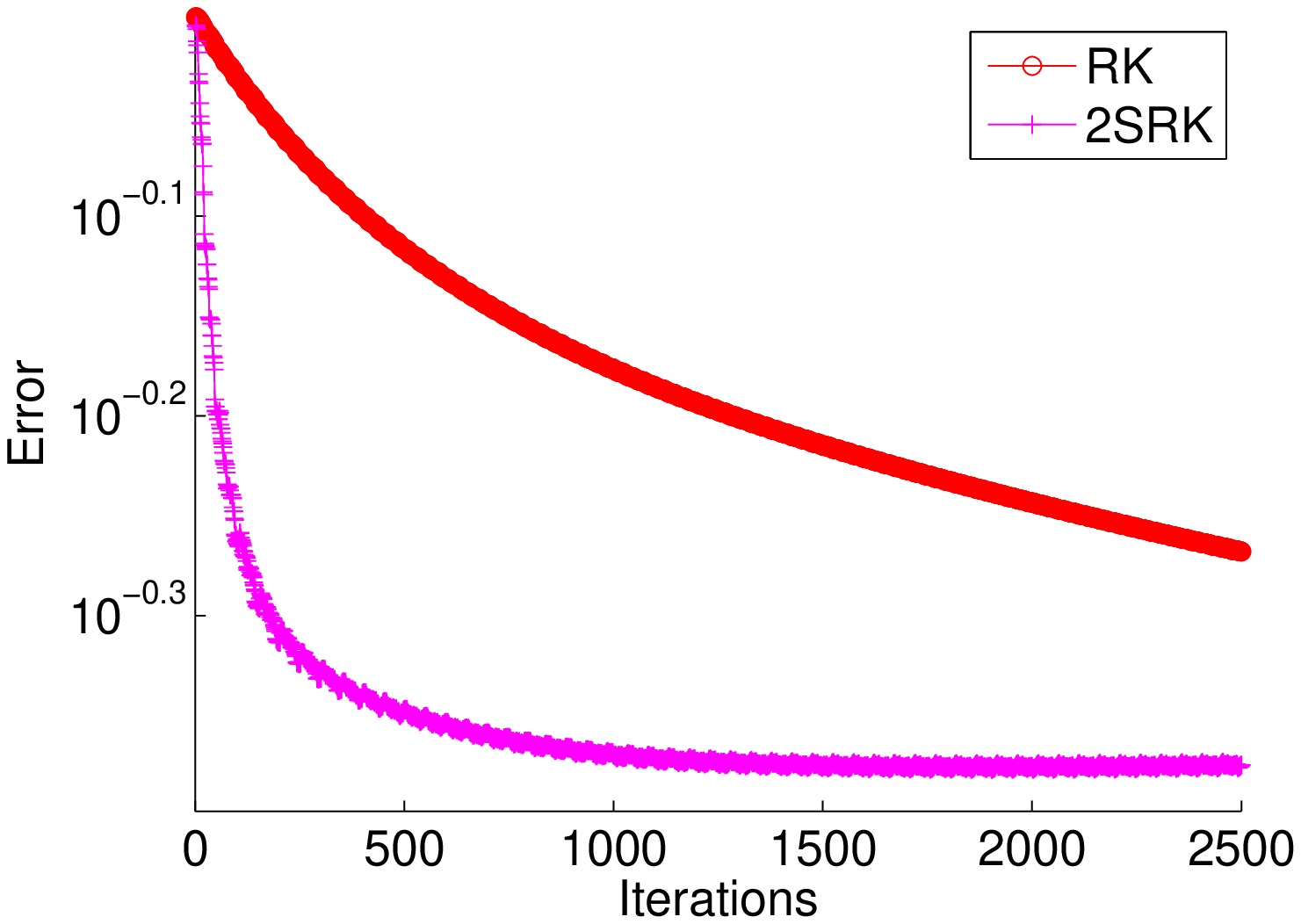} &
(d) \includegraphics[width=2.5in]{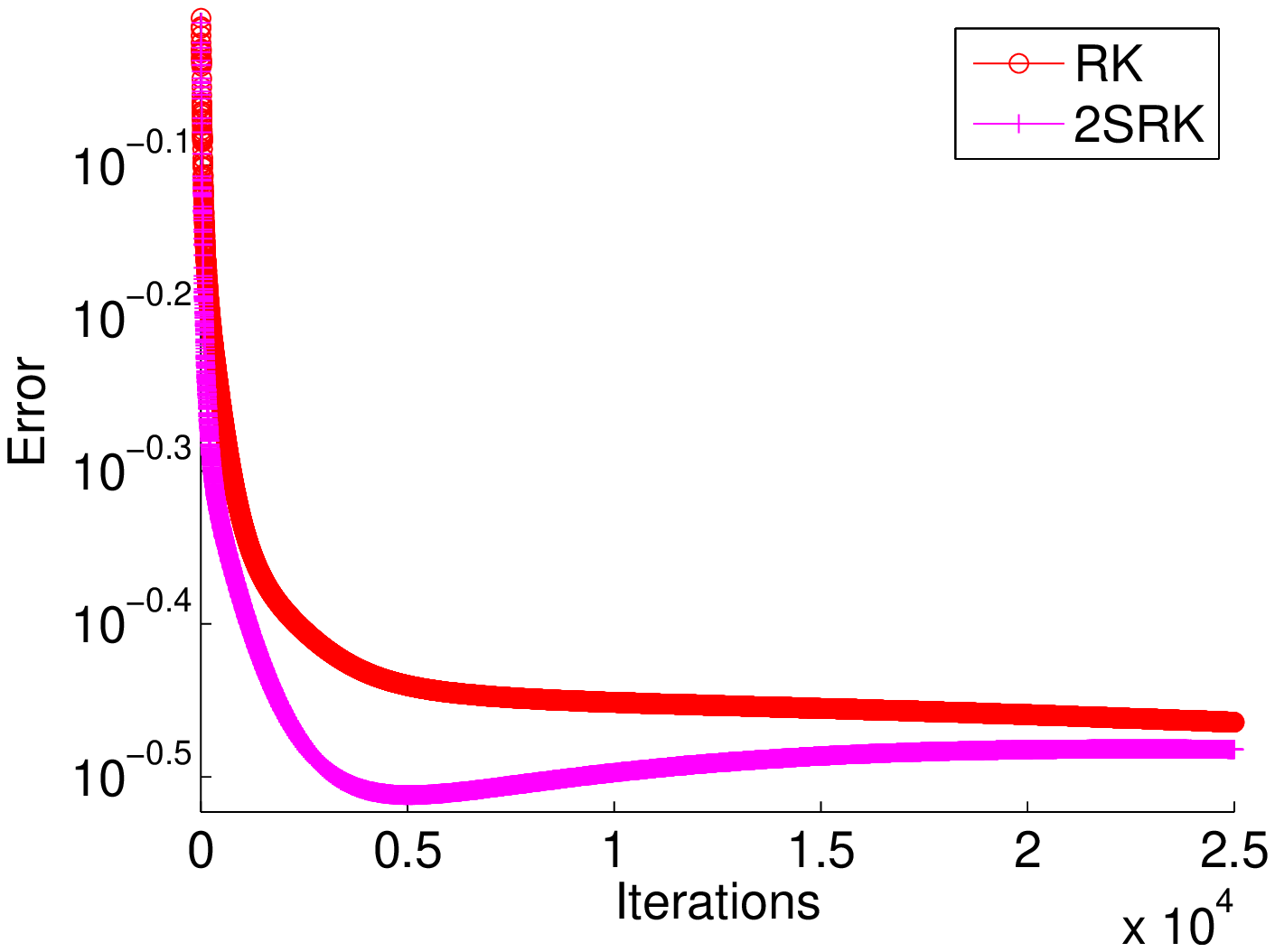}\\
\end{array}$    
\end{center}
\caption{A log-linear plot of the error per iteration for the randomized Kaczmarz (RK) and two-subspace RK (2SRK).  Matrix $\mtx{A}$ has uniformly distributed highly coherent rows with (a) $\delta = 0.651$ and $\Delta = 0.932$, (b) $\delta = 0.933$ and $\Delta = 0.985$, (c) $\delta = 0.964$ and $\Delta = 0.992$, and (d) $\delta = 0.965$ and $\Delta = 0.991$. }\label{fig:noise1}
\end{figure}

Plots (a), (b), and (c) demonstrate exponential convergence to the error threshold (or below) with improvements over the standard method.  Plot (d) shows the semi-convergence effect of the two-subspace Kaczmarz method.  It is an open problem to determine at which point to terminate the method for optimal error without knowledge of the solution $\vct{x}$.  One option is to simply to terminate after $\bigO(n^2)$ iterations, as this is the amount of iterations needed for convergence (see the discussion in~\cite{SV09:Arand}).  This is of course not optimal, and as Figure~\ref{fig:noise2} shows, using this halting criterion may cause the two-subspace method to perform worse than the standard Kaczmarz method.  We leave overcoming this challenge for future work. 

  \begin{figure}[h!]
\begin{center}
$\begin{array}{c@{\hspace{.1in}}c}
(a) \includegraphics[width=2.5in]{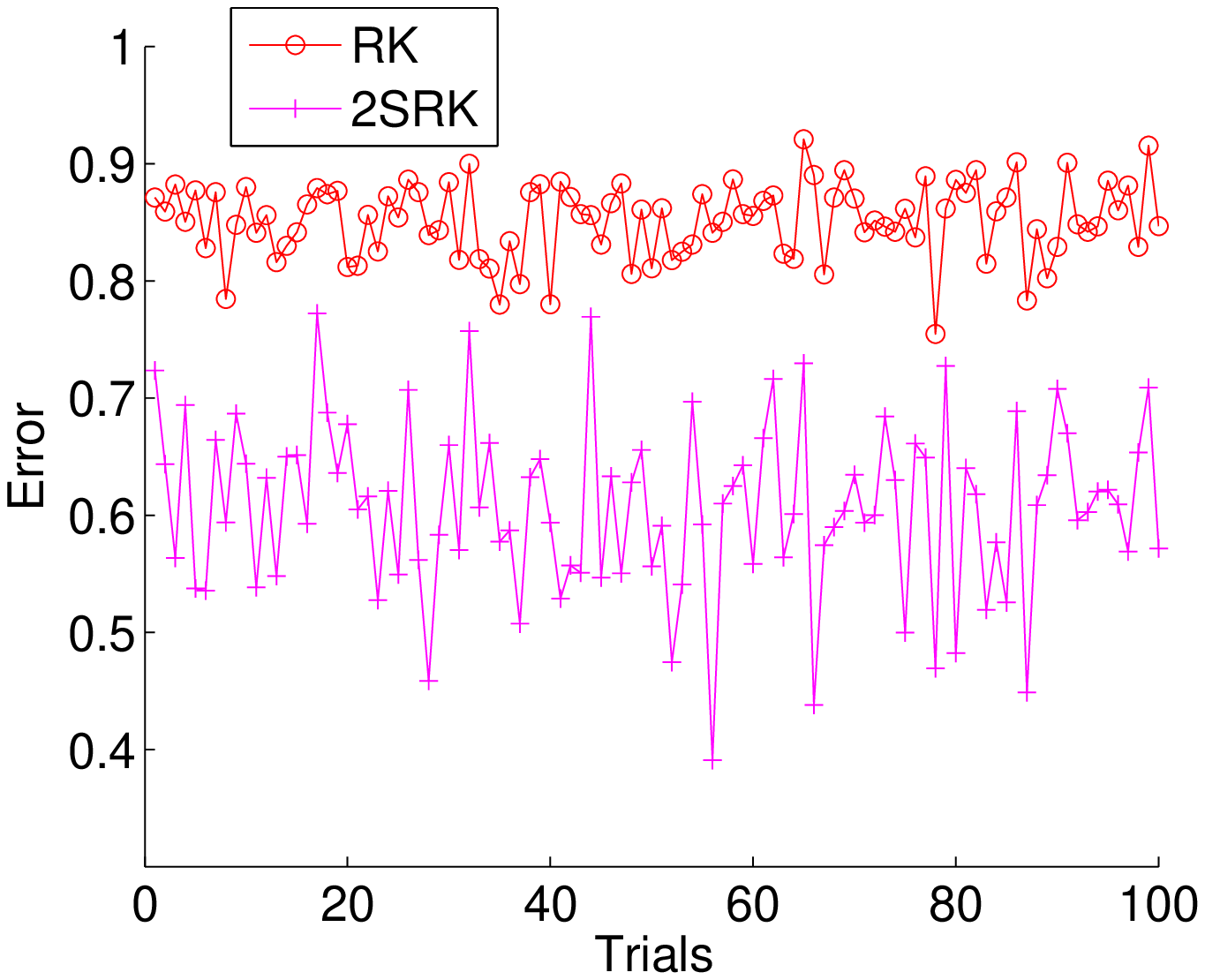}  &
(b) \includegraphics[width=2.5in]{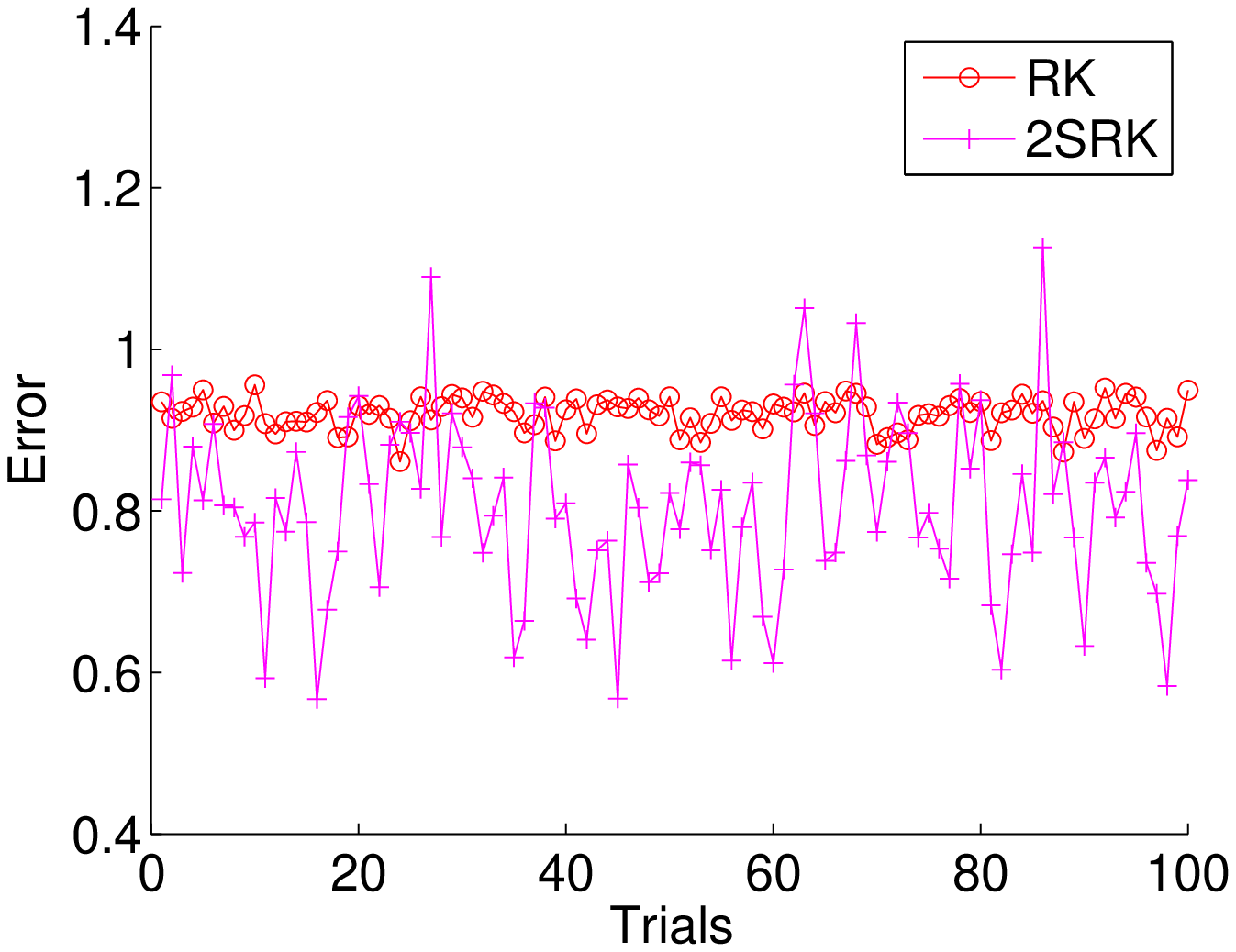}  \\
\end{array}$    
\end{center}
\caption{A plot of the error after 500 iterations for the randomized Kaczmarz (RK) and two-subspace RK (2SRK).  Matrix $\mtx{A}$ has uniformly distributed highly coherent rows with (a) $\delta = 0.981$ and $\Delta = 0.995$, and (b) $\delta = 0.993$ and $\Delta = 0.998$. }\label{fig:noise2}
\end{figure}

\section{Discussion}\label{sec:discuss}

As is evident from Theorems~\ref{thm:main} and~\ref{thm:improve}, the two-subspace Kaczmarz method provides exponential convergence in expectation to the solution of $\mtx{A}\vct{x} = \vct{b}$.  The constant in the rate of convergence for the two-subspace Kaczmarz method is at most equal to that of the best known results for the randomized Kaczmarz method~\eqref{SV}.  When the matrix $\mtx{A}$ has many correlated rows, the constant is significantly lower than that of the standard method, yielding substantially faster convergence.  This has positive implications for many applications such as nonuniform sampling in Fourier analysis, as discussed in Section~\ref{sec:intro}. 

We emphasize that the bounds presented in our main theorems are weaker than what we actually prove, and that even when $\delta$ is small, if the rows of $\mtx{A}$ still have many correlations, Lemmas~\ref{lem:main} and~\ref{lem:improve} still guarantee improved convergence. For example, if the matrix $\mtx{A}$ has correlated rows but contains a pair of identical rows and a pair of orthogonal rows, it will of course be that $\delta = 0$ and $\Delta = 1$.  However, we see from the proofs of our main theorems that the two-subspace method still guarantees substantial improvement over the standard method.  Numerical experiments in cases like this produce results identical to those in Section~\ref{sec:numerics}. 

It is clear both from the numerical experiments and Theorem~\ref{thm:main} that the two-subspace Kaczmarz performs best when the correlations $\<\vct{a_r}, \vct{a_s}\>$ are bounded away from zero.  In particular, the larger $\delta$ is the faster the convergence of the two-subspace method.  The dependence on $\Delta$, however, is not as straightforward.  Theorems~\ref{thm:main} and~\ref{thm:improve} suggest that when $\Delta$ is very close to $1$ the two-subspace method should provide similar convergence to the standard method.  However, in the experiments of Section~\ref{sec:numerics} we see this is not the case.  This dependence on $\Delta$ appears to be only an artifact of the proof.  

\subsection{Noisy systems}
As is the case for many iterative algorithms, the presence of noise introduces complications both theoretically and empirically.  Theorem~\ref{thm:noise} guarantees expected exponential convergence to the noise threshold.  For pessimistic values of $D$, the noise threshold provided by Theorem~\ref{thm:noise} is greater than that of the standard method,~\eqref{eq:noise}, by a factor of $\sqrt{R}$.  In addition, large values of $\Delta$ produce large error thresholds in this bound.  As in the noiseless case, we believe this dependence on $\Delta$ to be an artifact of the proof. 

A further and important complication that noise introduces is a semi-convergence effect, a well-known effect in Algebraic Reconstruction Technique (ART) methods (see e.g.~\cite{ENP10:semi}).  For example, in Figure~\ref{fig:noise1} (d), the estimation error for the two-subspace method decreases to a point and then begins to increase.  It remains an open problem to determine an optimal stopping condition without knowledge of the solution $\vct{x}$.  

\subsection{Future Work}
The issue of detecting semiconvergence is a very deep problem.  The simple solution would be to terminate the algorithm once the residual $\|\mtx{A}\vct{x_k} - \vct{b}\|_2$ decreases below some threshold.  However, the residual decreases in each iteration even when the estimation error begins to increase.  Determining the residual threshold beyond which one should terminate is not an easy problem and work in this area continues to be done. 

We also hope to improve the error threshold bound of Theorem~\ref{thm:noise} for the two-subspace method.  We conjecture that the $(1-\Delta)$ term can be removed or improved, and that the dependence on $R$ can be reduced to $\sqrt{R}$ in the error term of Theorem~\ref{thm:noise}.  

Finally, a natural extension to our method would be to use more than two rows in each iteration.  Indeed, extensions of the two-subspace algorithm to arbitrary subspaces can be analyzed~\cite{NT12:paving}. 

\bibliography{rk}

\end{document}